\documentclass[submission]{dmtcs-episciences}

\usepackage{hyperref}

\usepackage{amsmath}
\usepackage{amsthm}
\usepackage{amssymb,epsfig}
\usepackage[utf8]{inputenc}
\usepackage{graphics}
\usepackage{graphicx}
\usepackage{xcolor}

\newtheorem{theorem}{Theorem}[section]
\newtheorem{lemma}[theorem]{Lemma}
\theoremstyle{definition}
\newtheorem{remark}[theorem]{Remark}
\newtheorem{definition}[theorem]{Definition}

\usepackage{mathdots}
\usepackage{algorithm}
\usepackage{algorithmic}
\usepackage{multirow}

\newcommand{\Zomega}{\mathbb{Z}[\omega]}
\newcommand{\Zbeta}{\mathbb{Z}[\beta]}

\newcommand{\NN}{\mathbb{N}}
\newcommand{\ZZ}{\mathbb{Z}}
\newcommand{\QQ}{\mathbb{Q}}
\newcommand{\RR}{\mathbb{R}}
\newcommand{\CC}{\mathbb{C}}

\newcommand{\A}{\mathcal{A}}
\newcommand{\B}{\mathcal{B}}
\newcommand{\C}{\mathcal{C}}
\newcommand{\D}{\mathcal{D}}
\newcommand{\Q}{\mathcal{Q}}
\newcommand{\R}{\mathcal{R}}

\newcommand{\Qw}[3][w]{\Q_{[#1_{-#2}, \dots, #1_{-#3}]}}
\newcommand{\Qwo}[2][w]{\Q_{[#1_{0}, \dots, #1_{-#2}]}}

\newcommand{\tuple}[3][w]{(#1_{-#2}, \dots, #1_{-#3})}
\newcommand{\tupleo}[2][w]{(#1_{0}, \dots, #1_{-#2})}

\newcommand{\vect}[1]{({#1}_0,{#1}_1,\cdots,{#1}_{d-1})^T}

\newcommand{\decdot}{\raisebox{0.1ex}{\textbf{.}}}

\newcommand{\Qb}[1]{\mathcal{Q}_{[\scriptstyle b]}^{\scriptstyle #1}}

\newcommand{\fin}[1]{\text{Fin}_{#1}(\beta)}

\newcommand{\vertiii}[1]{{\left\vert\kern-0.25ex\left\vert\kern-0.25ex\left\vert #1\right\vert\kern-0.25ex\right\vert\kern-0.25ex\right\vert}}

\newcommand{\norm}[2]{\left\lVert#1\right\rVert_{#2}}
\newcommand{\Mnorm}[2]{\vertiii{#1}_{#2}}
\newcommand{\normBeta}[1]{\norm{#1}{\beta}}

\usepackage{pifont}

\title{Construction of Algorithms for Parallel Addition}

\author{Jan Legersk\'y\affiliationmark{1,2} \and Milena Svobodov\'a\affiliationmark{2}}

\affiliation{Research Institute for Symbolic Computation, Johannes Kepler University Linz, Austria \\
	Faculty of Nuclear Sciences and Physical Engineering, Czech Technical University in Prague, Czech Republic}

 
\begin{document}
\publicationdetails{VOL}{2018}{ISS}{NUM}{SUBM}
\maketitle
\begin{abstract}
An algebraic number $\beta \in \CC$ with no conjugate of modulus~$1$ can serve as the base of a numeration system $(\beta, \A)$ with parallel addition, i.e., the sum of two operands represented in base $\beta$ with digits from~$\A$ is calculated in constant time, irrespective of the length of the operands. In order to allow parallel addition, sufficient level of redundancy must be given to the alphabet~$\A$. The complexity of parallel addition algorithm depends heavily on the size $\#\A$ of the alphabet: the bigger alphabet is considered, the lower complexity of the parallel addition algorithm may be reached, and vice versa.\\
Here we aim to find parallel addition algorithms on alphabets of the minimal possible size, for a given base. As the complexity of these algorithms becomes quite huge in general, we introduce a so-called {\it Extending Window Method} (EWM) -- in fact an algorithm to construct parallel addition algorithms. This method can be applied on bases~$\beta$ which are expanding algebraic integers, i.e., $\beta$~whose all conjugates are greater than~$1$ in modulus. Convergence of the EWM is not guaranteed, nevertheless, we have developed tools for revealing non-convergence, and there is a~number of successful applications. Firstly, the EWM provides the same parallel addition algorithms as were previously introduced by A.~Avizienis, C.Y.~Chow \& J.E.~Robertson or B.~Parhami for integer bases and alphabets. Then, by applying the EWM on selected complex bases~$\beta$ with non-integer alphabets $\A \subset \Zbeta$, we obtain new results -- parallel addition algorithms on alphabets of minimal possible size, which could not be found so far (manually). The EWM is helpful also in the case of block parallel addition.
\end{abstract}

\section{Introduction}\label{sec:introduction}

Numeration systems $(\beta, \A)$ are defined by a base~$\beta$ (sometimes also called a radix) and an alphabet~$\A$ -- the set of digits. Usually, the base is greater than~$1$ in modulus: $|\beta| > 1$. A number~$x$ is said to have a representation in the numeration system $(\beta, \A)$, if there exist digits $x_j$ from the alphabet $\A$ such that $x = \sum_{j=r}^l x_j \beta^j$, with $l \in \ZZ$ and $r \in (\ZZ \cup \{-\infty\})$. The $(\beta, \A)$-representation of~$x$ is then denoted by the string $(x_l \cdots x_0 \decdot x_{-1} \cdots)_{\beta}$, and it may not be unique.\\

Our goal is to find those numeration systems $(\beta, \A)$ that allow parallel addition, i.e., where the operation of adding two numbers having $(\beta, \A)$-representations has constant time complexity, assuming that arbitrary number of operations can be run in parallel. This concept was introduced by A.~Avizienis in~\cite{Avizienis}, and then was intensively elaborated and brought into practice -- not just for the sake of the addition operation itself, but also as part of other calculations (e.g. fast algorithms for multiplication or division). In practice, we actually consider only summands with finite $(\beta, \A)$-representations as the operands of addition algorithms.\\

When approaching the problem of parallel addition, we start by identifying the bases~$\beta$ for which parallel addition is possible at all, in a broad scope -- so not only integer, but also real and complex bases. The requirement to allow parallel addition forces the base to be an algebraic number (when considering $\A \subset \ZZ$ or $\A \subset \Zbeta$). Existence of parallel addition algorithm for a given base depends on the fact whether any of its algebraic conjugates is of modulus equal to~$1$ -- if not, then parallel addition algorithms exist for that base (with a suitably chosen alphabet~$\A$). But if $\beta$~has any algebraic conjugate equal to~$1$ in modulus, then parallel addition in base~$\beta$ is impossible.\\

In the second step, given a base~$\beta$ suitable for parallel addition, we focus on the alphabet~$\A$. In principle, the more elements in the alphabet, the better chances for the numeration system $(\beta, \A)$ to allow parallel addition -- as the redundancy increases, and it is the redundancy that makes parallel addition possible. In~\cite{FrPeSv}, a constructive method is provided, describing how to choose the alphabet~$\A$ (of consecutive integers, containing~$0$, and symmetric), and then how to perform parallel addition in thus obtained numeration system $(\beta, \A)$. The approach used there consists in finding a~$(\beta, \A)$-representation of zero with one dominant digit, i.e., strictly greater than the sum of moduli of all the other digits. Such representation of zero then serves as the core rewriting rule within the parallel addition algorithm, and also determines the size of the alphabet $\A \subset \ZZ$.\\

The drawback of this approach is that the alphabets~$\A$ are very large, and so is often also the carry propagation (due to the length of the rewriting rule). Too much redundancy may be inconvenient for other operations -- e.g. division or comparison. Therefore, our goal is to find alphabets of minimal size allowing parallel addition. For~$\beta$ being an algebraic integer, lower bounds on the cardinality of integer alphabets~$\A$ for parallel addition were specified in~\cite{FrPeSv2}; and, for several classes of the bases~$\beta$, the parallel addition algorithms using (integer) alphabets of the minimal size were actually found (\cite{FrPeSv2, FrHePeSv}).\\

Although the alphabet~$\A$ is most often selected as a subset of integers, the problem is relevant also in a more general setting, namely considering $\A \subset \ZZ[\beta]$. Some of the results and conditions on the alphabet~$\A$ to enable parallel addition, obtained earlier for integer alphabets, were generalized in~\cite{JLe-article} also to $\A \subset \Zbeta$. Especially the lower bound on minimal cardinality of~$\A$ in the case of $\A[\beta] = \Zbeta$. It is proven that the alphabet must contain all congruence classes modulo~$\beta$ in~$\ZZ[\beta]$, and also all congruence classes modulo $\beta-1$ in~$\ZZ[\beta]$. This is already quite an important requirement to fulfil, but still leaving freedom when selecting the digits to compose the set~$\A$. For various reasons, such as decreasing the computing demands for operations in $(\beta, \A)$, or to enable (on-line) division, we may propose the alphabet~$\A$ to be symmetric, or to contain the smallest possible digits in modulus.\\

In the sequel, we first recall the basic terminology and known results about parallel addition, in Section~\ref{sec:preliminaries}. Then, in Section~\ref{sec:addition}, we stress the difference between standard and parallel addition, and fix the notation to be used within the algorithms later. Once having a hypothesis, namely a set~$\A \subset \ZZ[\beta]$ as the candidate for $(\beta, \A)$ to become a numeration system with parallel addition, we use a~so-called Extending Window Method (EWM), which tries, in an automated and systematic way, to derive an algorithm of parallel addition for this numeration system $(\beta, \A)$. This idea of the EWM (a new result), together with schemes of its main algorithms are presented in Section~\ref{sec:EWM-description}.\\

Usage of the Extending Window Method is limited to bases~$\beta \in \CC$ being algebraic integers, and at the same time also \emph{expanding}, i.e., their conjugates must be all greater than~$1$ in modulus. Convergence of the EWM is guaranteed only for its first phase (generation of the weight coefficients set); whereas for the second phase (assignment of the weight coefficients to all necessary combinations of input digits), we developed tools for revealing non-convergence. This is explained in Section~\ref{sec:EWM-convergence}. Nevertheless, despite these limitations, the EWM is a valuable tool in the search for parallel addition algorithms, because the attempts to derive the algorithms manually are very laborious -- in fact close to impossible -- already for quite small alphabets (e.g. with less than 10 digits).\\

Section~\ref{sec:EWM-implementation} discusses generalization of EWM to bases and alphabets in $\Zomega$ for an algebraic integer $\omega$. It also summarizes the basic information about the actual implementation of the EWM done in SageMath.\\

Although the EWM does not guarantee finding the parallel addition algorithms in all cases where they exist, it brings a significant number of successful results; especially when considering non-integer alphabets $\A \subset \ZZ[\beta]$, and also for block parallel addition. In Section~\ref{sec:EWM-results}, we first show that in many (but not all) cases, the EWM provides the same parallel addition algorithms as derived earlier manually. Then, we give examples of new results -- parallel addition algorithms obtained only via EWM.\\

\section{Preliminaries and Known Results}\label{sec:preliminaries}

\subsection{Concept of Parallel Addition}\label{sub-sec:parallel-addition}

The idea of parallel addition has been formalized using the notion of so-called \emph{$p$-local function}:

\begin{definition}\label{def:plocal}
Let $\A$ and $\B$ be alphabets.
A \emph{function} $\varphi:\B^\ZZ \rightarrow \A^\ZZ$ is said to be \emph{$p$-local}
if there exist $r,t\in\NN$ satisfying $p=r+t+1$ and a function $\phi: \B^p \rightarrow \A$ such that,
for any $w=(w_j)_{j\in\ZZ}\in\B^\ZZ$ and its image $z=\varphi(w)=(z_j)_{j\in\ZZ}\in\A^\ZZ$,
we have $z_j=\phi(w_{j+t},\dots,w_{j-r})$ for every $j\in\ZZ$.
The parameters~$t$ and $r$ are called \emph{anticipation} and \emph{memory}, respectively.
\end{definition}

This means that a window of length $p$ computes the digit $z_j$ of the image $\varphi(w)$ from $(w_{j+t},\dots,w_{j-r})$, the digit $z_{j-1}$ from digits $(w_{j+t-1},\dots,w_{j-r-1})$ etc.\\

Since two $(\beta,\A)$-representations may be easily summed up digit-wise in parallel, the crucial point of parallel addition is conversion of a $(\beta,\A+\A)$-representation of the sum to a $(\beta,\A)$-representation. The notion of $p$-local function is applied to this conversion.

\begin{definition}\label{def:digitSetConversion}
Let $\beta$ be a base and let $\A$, $\B$ be alphabets containing~$0$. A function $\varphi:\B^\ZZ\rightarrow \A^\ZZ$ such that
    \begin{enumerate}[i)]
        \item for any $w=(w_j)_{j\in\ZZ}\in\B^\ZZ$ with finitely many non-zero digits, $z=\varphi(w)=(z_j)_{j\in\ZZ}\in\A^\ZZ$ has only finite number of non-zero digits, and
        \item $\sum_{j\in\ZZ} w_j \beta^j= \sum_{j\in\ZZ} z_j \beta^j$
    \end{enumerate}
is called \emph{digit set conversion} in the base $\beta$ from $\B$ to $\A$. Such a conversion $\varphi$ is said to be \emph{computable in parallel} if $\varphi$ is a $p$-local function for some $p\in\NN$. \emph{Parallel addition} in a numeration system $(\beta, \A)$ is a digit set conversion in base~$\beta$ from~$\A+\A$ to~$\A$ which is computable in parallel.
\end{definition}

\subsection{Parallel Addition Algorithms in Minimal Alphabets - Known Results}\label{sub-sec:minimal-alphabets}

Naturally, the first class of bases where the concept of parallel addition was introduced and studied are integers. In~\cite{Avizienis}, A.~Avizienis gave an algorithm for parallel addition in numeration systems with positive integer bases $\beta \in \NN$, $\beta \geq 3$ and with (symmetric) integer alphabets $\A = \{-a, \ldots, 0, \ldots, a \}$, such that $\beta/2 < a \leq \beta-1$. The size of thus prescribed alphabets for parallel addition equals $\#\A = 2a+1$, which is at least $\beta+2$ for $\beta$ odd and $\beta+3$ for $\beta$ even. So the redundancy is a bit higher here (+2 and +3 digits for odd and even bases, respectively) than the minimum needed for parallel addition (just +1 digit on top of the canonical alphabet); but, in turn, the algorithm performing parallel addition is very simple, and the carry propagates by just one position.\\

The level of redundancy in $(\beta, \A)$ was decreased to the minimum for even positive integer bases $\beta = 2a$ (including also $\beta = 2$) in the parallel addition algorithm provided in~\cite{ChowRobertson}; using again symmetric alphabets, i.e., $\A = \{ -a, \ldots, 0, \ldots, a \}$ with $\# \A = 2a+1 = \beta+1$. And in~\cite{Parhami}, B.~Parhami gave the algorithms for all positive integer bases $\beta \geq 2$, both even and odd, on any (generally non-symmetric) alphabets of consecutive integers (containing zero) of the minimal size $\#\A = \beta+1$.\\

Parallel addition algorithms on (integer) alphabets of the minimal size are provided in~\cite{FrPeSv2} and~\cite{FrHePeSv} for several classes of the bases:
\begin{itemize}
    \item positive and negative integers $\beta = \pm b$, with $b \in \NN, b \geq 2$: on $\# \A = b+1$;
    \item positive and negative rational numbers $\beta = \pm a/b$, with $a, b \in \NN$ co-prime: on $\# \A = a+b$;
    \item quadratic Pisot numbers $\beta > 1$ fulfilling $\beta^2 = a\beta - b$, with $a, b \in \NN$, $a-2 \geq b \geq 1$: on $\# \A = a+b-1$;
    \item quadratic Pisot numbers $\beta > 1$ fulfilling $\beta^2 = a\beta + b$, with $a, b \in \NN$, and $a-1 \geq b \geq 2$ or $a \geq b = 1$: on $\# \A = a+b+1$;
    \item roots of $\beta = \sqrt[\ell]{b}$, with $\ell, b \in \NN$, $b \geq 2$: on $\# \A = b+1$ (where $\beta$ cannot be written as $\beta = \sqrt[\ell'']{c}$, with $\ell = \ell' \ell''$, $b = c^{\ell'}$, and $\ell', \ell'', c \in \NN$).
\end{itemize}

As for the bases being roots of $\beta = \sqrt[\ell]{-b}$, with $b \in \NN$, $b \geq 2$, no general formula of parallel addition algorithm on alphabets of minimal size has yet been provided for this class as a whole, but just for selected individual cases (studied in~\cite{Nielsen+Muller, FrPeSv2}), again on integer alphabets:
\begin{itemize}
    \item Penney base $\beta = \imath - 1$: on $\# \A = 5$;
    \item Knuth base $\beta = 2 \imath$: on $\# \A = 5$;
    \item base $\beta = \imath \sqrt{2}$: on $\# \A = 3$.
\end{itemize}

\subsection{Calculating in Blocks}\label{sub-sec:k-block}

In~\cite{Kornerup}, an alternative approach was introduced to perform calculations on $(\beta, \A)$-representations, namely calculating in so-called $k$-blocks. It means that, for a given positive integer $k \geq 1$, we divide the $(\beta, \A)$-representation into blocks of $k$ digits, and then treat it in fact as a $(\beta^k, \B)$-representation, where the alphabet $\B = \{ \sum_{j=0}^{k-1} a_j \beta^j \, \vert \, a_j \in \A \}$ contains new digits with respect to new base~$\beta^k$.\\

As to the problem of parallel addition, we explain in Section~\ref{sub-sec:necessary-conditions} that the $k$-block concept does not broaden the set of eligible bases. Nevertheless, it may help to decrease the size of alphabets (level of redundancy) needed for parallel addition. This is the case e.g. of the $d$-bonacci bases -- roots of the minimal polynomials $X^d = X^{d-1} + X^{d-2} + \cdots + X + 1$ for $d \geq 2$: here the block concept decreases the size of the alphabet for parallel addition to just $3$ digits, instead of the $d+1$ digits necessary for $1$-block parallel addition.\\

For the so-called Canonical Number Systems (CNS, introduced by B.~Kov\'acs in~\cite{Kovacs}, and then extensively studied by others), with a complex base $\beta$ and alphabet $\{ 0, 1, \ldots, |N(\beta)|-1 \}$, where $|N(\beta)|$ denotes the norm of $\beta$ over $\QQ$, it is proved in~\cite{FrHePeSv} that block parallel addition is possible on the alphabets $\{ 0, 1, \ldots, 2|N(\beta)|-2 \}$ or $\{ -|N(\beta)|+1, \ldots, -1, 0, 1, \ldots, |N(\beta)|-1 \}$. The size of these alphabets ($\#\A = 2|N(\beta)|-1$) may be substantially smaller than the minimal size of alphabet needed for $1$-block parallel addition in the same base $\beta$. For instance, see Section~\ref{sec:EWM-results}:
\begin{itemize}
    \item Penney base $\beta = \imath -1$ requires a $5$-digit alphabet for ($1$-block) parallel addition, but only $3$~digits for $k$-block parallel addition: an algorithm using $k=4$ is provided in~\cite{Herreros}, and we further diminish the block size to $k=2$ in this work;
    \item Eisenstein base $\beta = \exp{(2\pi\imath/3)}$ requires a $7$-digit alphabet for ($1$-block) parallel addition, but only $5$~digits for $k$-block parallel addition: we find an algorithm using $k=3$.
\end{itemize}

\subsection{Necessary Conditions on Bases and Alphabets for Parallel Addition}\label{sub-sec:necessary-conditions}

Let us recall that we still consider just finite $(\beta, \A)$-representations as the operands (summands to be added):
\begin{equation}
    \fin{\A} = \left\{ x = \sum_{j=R}^L x_j \beta^j \, \vert \, x_j \in \A, L, R \in \ZZ \right\} \, ,
\end{equation}
and let us denote
\begin{equation}
    \A[\beta] = \left\{ \sum_{i=0}^N a_i \beta^i \colon a_i \in \A, \, N \in \NN \right\}
    \qquad \mathrm{and} \qquad
    \ZZ[\beta] = \left\{ \sum_{i=0}^N a_i \beta^i \colon a_i \in \ZZ, \, N \in \NN \right\} \, .
\end{equation}

For the case of alphabets $\A \subset \ZZ$ of consecutive integers, the condition on base~$\beta$ to allow parallel addition was proved in~\cite{FrHePeSv}, and later in~\cite{JLe-article} it was generalized to alphabets $\A \subset \Zbeta$:\\

\begin{theorem}\label{thm:base_no-modulo-1}
Let $\beta$ be a complex number such that $|\beta|>1$. There exists an alphabet $\A \subset \ZZ[\beta]$ with $0 \in \A$ and $1 \in \fin{\A}$ which allows $k$-block parallel addition in $(\beta,\A)$ for some $k \in \NN$, if and only if $\beta$ is an algebraic number with no conjugate of modulus~$1$. If this is the case, then there also exists an alphabet of consecutive integers containing~$0$ which enables $1$-block parallel addition in base~$\beta$.
\end{theorem}

The problem of the minimal alphabet size for parallel addition was again studied first for the case of alphabets of consecutive integers, e.g. in~\cite{FrPeSv2}, and then extended in~\cite{JLe-article} to the more general case of $\A \subset \Zbeta$. Those results provide not just the lower bound on the alphabet size, but also the form of digits that have to be contained in the alphabet, based on congruence classes modulo $\beta$ and $\beta-1$.
We recall that $\gamma,\delta\in\Zbeta$ are congruent modulo $\alpha\in\Zbeta$ if and only if $\gamma - \delta \in \alpha\cdot\Zbeta$.

\begin{theorem}\label{thm:modulo-beta&beta-1}
If a numeration system $(\beta, \A)$ with $\A[\beta] = \ZZ[\beta]$ allows parallel addition, then the alphabet~$\A \subset \Zbeta$ contains at least one representative of each congruence class modulo~$\beta$ and modulo~$\beta-1$ in~$\Zbeta$.
\end{theorem}

In the case of bases~$\beta$ being algebraic integers, it is convenient that the lower bound on the size of alphabets for parallel addition can be expressed by means of the minimal polynomial of~$\beta$:

\begin{theorem}\label{thm:minimal-alphabet-size}
Let $(\beta,\A)$ be a numeration system such that $\beta \in \CC, |\beta| > 1$ is an algebraic integer with minimal polynomial $m_{\beta}$, and let $\A[\beta] = \ZZ[\beta]$. If $(\beta, \A)$ allows parallel addition, then $\beta$~is expanding and
$$ \#\A \geq \max \{|m_\beta(0)|, |m_\beta(1)|\}\,. $$
Moreover, if $\beta$ has a positive real conjugate, then
$$ \#\A \geq \max \{|m_\beta(0)|, |m_\beta(1)|+2\}\,. $$
\end{theorem}

For completeness, let us state also the earlier result from~\cite{FrPeSv2}, which is a bit stronger than Theorem~\ref{thm:minimal-alphabet-size} for the alphabets of consecutive integers, as it does not require that $\A[\beta] = \ZZ[\beta]$:

\begin{theorem}\label{thm:base_alg-int_positive}
Consider a base~$\beta \in \CC, |\beta| > 1$, being an algebraic integer with minimal polynomial~$m_{\beta}$. Let $\A$~be an alphabet of consecutive integers containing~$0$ and~$1$. If addition in ${\rm Fin}_{\A}(\beta)$ is computable in parallel, then $\# \A \geq |m_{\beta}(1)|$. If, moreover, the base has a positive real conjugate, then $\# \A \geq |m_{\beta}(1)| + 2$.
\end{theorem}

\section{Addition -- Standard vs. Parallel}\label{sec:addition}

The general idea of addition (standard or parallel) in any numeration system $(\beta,\A)$ is the following: we sum up two numbers digit-wise, and then convert the result with digits in~$\A+\A$ into the alphabet~$\A$. Obviously, digit-wise addition is computable in parallel, so the problematic part is the digit set conversion of thus obtained result. It can be easily done in a standard way (from right to left), but parallel digit set conversion is non-trivial. Parallel conversion may be based on the same formulas as the standard one, but the choice of so-called \emph{weight coefficients} differs in general.\\

Let $w_{n'} \cdots w_1 w_0 \bullet w_{-1} \cdots w_{-m'}$ be a $(\beta, \A+\A)$-representation of $w \in \fin{\A+\A}$ obtained by digit-wise addition of $(\beta,\A)$-representations of summands. We search for a $(\beta,\A)$-representation of~$w$, i.e., a sequence $z_{n} \cdots z_1 z_0 z_{-1} \cdots z_{-m}$ such that $z_j \in \A$ and $ z_{n} \cdots z_1 z_0 \bullet z_{-1} \cdots z_{-m}=(w)_{\beta,\A}$. Note that the indices $n$ and $-m$ of the first and the last non-zero digits of the converted representation $(w)_{\beta,\A} = z_{n} \cdots z_1 z_0 \bullet z_{-1} \cdots z_{-m}$ generally differ from the indices $n'$ and~$-m'$ of the original representation $(w)_{\beta,{\A+\A}} = w_{n'} \cdots w_1 w_0 \bullet w_{-1} \cdots w_{-m'}$. We have $ n\geq n'$ and $m \geq m'$; if not, the converted representation is padded by zeros. For easier notation, such representations can be multiplied by~$\beta^{m'}$. Hence, without loss of generality, we can consider only conversion of elements of~$(\A+\A)[\beta]$, i.e., numbers from~$\fin{\A+\A}$, whose representations have all digits with negative indices equal to zero.\\

Digits $w_j$ are converted from $\A+\A$ into the alphabet~$\A$ by digit-wise addition of suitable representations of zero. Any polynomial $R(x)=r_s x^s + \dots + r_1 x + r_0$ with coefficients $r_j \in \Zbeta$ such that $R(\beta) = 0$ gives a representation of zero in the base~$\beta$. Such polynomial $R$ is called a \emph{rewriting rule} for~$\beta$. One of the coefficients of~$R$ which is greatest in modulus (so-called \emph{dominant coefficient}) may be used for conversion of a digit from~$\A+\A$ into~$\A$. Nevertheless, the Extending Window Method proposed later in Section~\ref{sec:EWM-description} to generate parallel addition algorithms is strongly dependent on the rewriting rule. Therefore, usage of an arbitrary rewriting rule~$R$ is not within the scope of this work, and we focus only on the simplest possible representation of zero -- rewriting rule deduced from the polynomial
\begin{equation}\label{eq:basic-zero-representation}
    R(x) = x - \beta \in \left(\Zbeta\right)[x]\,.
\end{equation}

Since $R(\beta) = 0 = \beta^{j} \cdot R(\beta) = 1 \cdot \beta^{j+1} - \beta \cdot \beta^{j}$ for any $j \in \NN$, there is a representation of zero in the form $1 (-\beta) 0 \cdots 0 \bullet = (0)_\beta$, with $-\beta$~on the $j$-th~position. We multiply this representation by a so-called \emph{weight coefficient} $q_j \in \Zbeta$, in order to obtain another representation of zero in the form
\begin{equation*}
    q_j (-q_j \beta) \underbrace{0 \cdots 0}_{j}\bullet = (0)_\beta \, .
\end{equation*}

This is digit-wise added to $w_{n} \cdots w_1 w_0 \bullet$ in order to convert the digit~$w_j$ from~$\A+\A$ into the alphabet~$\A$. Such conversion of the $j$-th~digit causes a \emph{carry}~$q_{j}$ onto the $(j+1)$-th~position, i.e., to the left neighbour. This is clear from the column notation of digit-wise addition.
\begin{center}
\begin{tabular}{rccccccclcl}
	$w_{n'}$              & $\cdots$  & $w_{j+1}$ & $w_{j}$        & $w_{j-1}$        & $\cdots$  & $w_1$ & $w_0$ & $\bullet$ & $=$ & $(w)_{\beta,{\A+\A}}$\\
	                      &           &           & $q_{j-1}$      & $-\beta q_{j-1}$ & $\iddots$ &       &       &           & $=$ & $(0)_\beta$\\
	                      & $\iddots$ & $q_{j}$   & $-\beta q_{j}$ &                  &           &       &       &           & $=$ & $(0)_\beta$ \\ \hline
    $z_{n} \cdots z_{n'}$ & $\cdots$  & $z_{j+1}$ & $z_{j}$        & $z_{j-1}$        & $\cdots$  & $z_1$ & $z_0$ & $\bullet$ & $=$ & $(w)_{\beta,{\A}}$\\
\end{tabular}
\end{center}

Hence, the desired formula for conversion on the $j$-th position is $z_j := w_j + q_{j-1} - q_j \beta$. The coefficient $q_{-1}$ is zero, since there is no carry from the right onto the $0$-th position. The terms \emph{carry} and \emph{weight coefficient} are related to a specific position: a weight coefficient~$q_{j-1}$ is a carry from the right neighbour $(j-1)$ onto the $j$-th position, $q_j$~is a weight coefficient chosen on the $j$-th position, and thus $q_j$~is a carry from the $j$-the position onto the $(j+1)$-th position, etc. Therefore, the conversion using the rewriting rule $x - \beta$ prolongs the part of non-zero digits only to the left, as there is no carry to the right. So all positions with negative indices remain with zero digits in the converted representation $z_n \cdots z_1 z_0 \bullet = (w)_{\beta,{\A}}$.\\

The conversion preserves the value of~$w$, since only representations of zero are added, formally
\begin{align}\label{eq:valuePreserving}
    \sum_{j\geq 0} z_j \beta^j &=w_0 - \beta q_0 + \sum_{j> 0} (w_j + q_{j-1} - q_j \beta) \beta^j  \notag = \\
    &= \sum_{j\geq 0} w_j \beta^j + \sum_{j>0} q_{j-1} \beta^j - \sum_{j\geq 0} q_j \cdot \beta^{j+1} = \\
    &=\sum_{j\geq 0} w_j \beta^j + \sum_{j>0} q_{j-1} \beta^j - \sum_{j> 0} q_{j-1} \cdot \beta^j = \sum_{j\geq 0} w_j \beta^j = w \,. \notag
\end{align}

The weight coefficients $q_j$ must be chosen so that the converted digits $z_j$ are in the alphabet~$\A$, namely,
\begin{equation}\label{eq:conversionFormula}
    z_j = w_j + q_{j-1} - q_j \beta \in \A \mathrm{\ for\ any\ } j \in \NN \,.
\end{equation}

For standard addition algorithm, the digit set conversion runs from the right ($j=0$) to the left ($j = n$), until all non-zero digits and carries are converted into the alphabet~$\A$. In this way, determination of the weight coefficients is a trivial task since the carry is known when a weight coefficient is determined.\\

But when designing parallel addition algorithms, choosing the weight coefficient is the crucial and difficult task, because there might be more possible carries from the right neighboring $(j-1)$-th~position that arrive to the $j$-th~position. This is done via the Extending Window Method, as described in Section~\ref{sec:EWM-description}. We require that the digit set conversion from~$\A+\A$ into~$\A$ is computable in parallel, i.e. there exist constants $r, t \in \NN_0$ such that for all $j\geq 0$ it holds that $z_j = z_j(w_{j+t}, \dots, w_{j-r})$. In our case, the anticipation~$t$ equals zero, since we use the rewriting rule $x - \beta$. To avoid dependency on all less significant digits (on the right from the processed position), we need some variety in the choice of the weight coefficient~$q_j$. This implies that the used numeration system $(\beta, \A)$ must be redundant.\\

The core difference between standard and parallel addition can be expressed, for conversion of $(\beta, \A + \A)$-representation $w_{n'} \ldots w_1 w_0 \bullet = (w)_{\beta, \A + \A}$ into $(\beta, \A)$-representation $z_{n} \ldots z_1 z_0 \bullet = (w)_{\beta, \A}$, as follows:
\begin{align*}
    \mathrm{standard\ addition:} \qquad z_j &=  z_j(w_{j}, \dots, w_{j-r}, \dots, w_{0})      \,, \\
    \mathrm{parallel\ addition:} \qquad z_j &=  z_j(w_{j+t}, \dots, w_{j-r})  \,.
\end{align*}

\begin{remark}\label{rem:A+A_B_A}
Consider a base $\beta \in \CC, |\beta| > 1$, and alphabet $\A \subset \Zbeta$ with sets $\B, \D \subset \Zbeta$ satisfying:
\begin{equation}
    0 \in \D \subset \A \subsetneq \B \subset \A + \A \qquad \mathrm{and} \qquad \A + \D \subset \B \, ,
\end{equation}
and let $l \in \NN$ be such that
\begin{equation}\label{eq:x=a+d+...+d}
    (\forall x \in \A) (\exists d^{(1)}, \ldots, d^{(l)} \in \D) (x = d^{(1)} + \cdots + d^{(l)}) \, .
\end{equation}
Then existence of parallel conversion in base~$\beta$ from~$\B$ to~$\A$ implies existence of parallel conversion in base~$\beta$ from~$\A + \A$ to~$\A$, and thus also existence of parallel addition in the numeration system $(\beta, \A)$.\\
Clearly, when summing up two elements $x, y \in \fin{\A}$ expressed as $x = \sum x_j \beta^j, y = \sum y_j \beta^j$ with $x_j, y_j \in \A$, we can use the form \eqref{eq:x=a+d+...+d} of each $x_j = d_j^{(1)} + \cdots + d_j^{(l)}$, and denote by $d^{(k)} = \sum d_j^{(k)} \beta^j$ for each $k = 1, \ldots, l$. Then, we split the operation $x + y$ into $l$~operations, gradually producing $z^{(k)} := d^{(k)} + z^{(k-1)}$ for $k = 1, \ldots, l$, with $z^{(0)} = y$. Each $k$-th step provides a~$(\beta, \A)$-representation of $z^{(k)}$ via parallel conversion of $d^{(k)} + z^{(k-1)}$ from~$\A + \D \subset \B$ to~$\A$. After $l$~iterations, we get the desired $(\beta, \A)$-representation of $z^{(l)}$ in the form
\begin{align*}
	z^{(l)} &= d^{(l)} + z^{(l-1)} = d^{(l)} + (d^{(l-1)} + z^{(l-2)}) = \cdots \\
	&= d^{(l)} + (d^{(l-1)} + (\cdots + (d^{(1)} + y))) = \sum_{k=1}^l d^{(k)} + y = x + y \, .
\end{align*}
The number $l \in \NN$ of iterations is fixed (not depending on $x, y \in \fin{\A}$), and composition of a fixed number of parallel conversions is still a parallel conversion; so we have the conversion from~$\A + \A$ to~$\A$ done in parallel.\\
This approach can be easily used for alphabets $\A \subset \ZZ$ of consecutive integers, e.g.:
\begin{itemize}
    \item $\A = \{m, \ldots, 0, \ldots, M\}$ with $\D = \{0, \pm 1\}$,  $\B = \{m-1, \ldots, M+1\}$ and $l = \max\{-m, M\}$; or
    \item $\A = \{0, \ldots, M\}$ with $\D = \{0, 1\}$, $\B = \{0, \ldots, M+1\}$ and $l = M$.
\end{itemize}
In the case of complex alphabets $\A \subset \Zbeta$, the application may be useful especially when bigger (non-minimal) alphabets are considered for the parallel addition algorithm.
\end{remark}

\section{Extending Window Method to Construct Parallel Addition Algorithms}\label{sec:EWM-description}

For a numeration system $(\beta, \A)$ such that $\beta$~is an algebraic integer and ${\A \subset \Zbeta}$, and let $\B \subset \Zbeta$ be an input alphabet with $\A \subsetneq \B \subset \A+\A$, according to Remark~\ref{rem:A+A_B_A}. The Extending Window Method (EWM) is a~newly proposed approach, attempting to construct algorithms for digit set conversion in the base~$\beta$ from~$\B$ to~$\A$ computable in parallel. Due to our choice~\eqref{eq:basic-zero-representation} of the representation of zero used within the EWM, the resulting parallel addition algorithms always have zero anticipation $t=0$, so carries only from the right.\\

As mentioned above, the key problem is to find appropriate weight coefficients $q_j \in \Zbeta$ such that
\begin{equation*}
	z_j = \underbrace{w_j}_{\in \B} + \,q_{j-1} - q_j \beta \in \A \mathrm{\quad for\ all \quad} j\geq 0 \, ,
\end{equation*}
for any input $w \in \fin{\B}$ with $(\beta, \B)$-representation $w_{n'} \dots w_1 w_0 \bullet = (w)_{\beta,\B}$.\\

The digits~$z_j$ of the result have to satisfy $z_j = z_j (w_{j}, \dots, w_{j-r})$ for some fixed memory $r \in \NN$, so the carries~$q_j$ have to fulfil $q_j = q_j (w_j, \ldots, w_{j-(r-1)})$ for the same~$r$. But the fact that the converted digit is on the $j$-th~position is not important -- the conversion must proceed in the same way on every position. Therefore, we simplify the notation by omitting the index~$j$ from the subscripts. From now on, $w_0 \in \B$ is the converted digit, $w_{-1} w_{-2} \dots \in \B$ are its neighboring digits on the right, $q_{-1} \in \Zbeta$ is the carry from the right, and we search for a weight coefficient $q_0 \in \Zbeta$ such that
\begin{equation*}
    z_0=w_0 + q_{-1} - q_0 \beta \in \A \,.
\end{equation*}

Before describing the Extending Window Method, let us introduce two definitions:

\begin{definition}\label{def:weightCoefficientsSet}
Let $(\beta, \A)$ be a numeration system, and let $\B \subset \Zbeta$ be a digit set such that $\A \subsetneq \B \subset \A+\A$. Any finite set $\Q \subset \Zbeta$ containing~$0$ such that
\begin{equation*}
    \B + \Q \subset \A + \beta \Q
\end{equation*}
is called a \emph{weight coefficients set} for the given numeration system $(\beta, \A)$ and input digit set~$\B$.
\end{definition}

A weight coefficients set $\Q \subset \Zbeta$ satisfies
\begin{equation}
    (\forall w_0 \in \B)(\forall q_{-1}\in\Q)(\exists q_0 \in \Q )(\underbrace{w_0 + q_{-1} - q_0 \beta}_{z_0} \in \A ) \,.
\end{equation}
In other words, there is a weight coefficient $q_0 \in \Q$ for any carry $q_{-1} \in \Q$ from the right and for any digit~$w_0$ from the input alphabet~$\B$, such that $z_0 = w_0 + q_{-1} - q_0 \beta$ is in the target alphabet~$\A$.

\begin{definition}\label{def:weightFunction}
Let $\Q \subset \Zbeta$ be a weight coefficients set for numeration system $(\beta, \A)$ and input digit set $\B \subset \Zbeta$. Let $r \in \NN$, and let $q:\B^{r} \rightarrow \Q$ be a mapping such that
\begin{align*}
    w_0 + q(w_{-1}, \dots, w_{-r}) - \beta q\tupleo{(r-1)} & \in    \A \text{ for\ any } w_0,w_{-1}, \dots, w_{-r} \in \B \, , \\
                                         \text{and} \quad   q(0, \dots, 0) & =      0 \,.
\end{align*}
Such mapping~$q$ is called \emph{weight function of length}~$r$ for $(\beta, \A)$ and input digit set~$\B$.
\end{definition}

Having a weight function $q: \B^{r} \rightarrow \Q$, we define a function $\phi: \B^{r+1} \rightarrow \A$ by the formula
\begin{equation}\label{eq:localConversion}
    \phi(w_{0}, \dots, w_{-r})=w_0+ \underbrace{q(w_{-1}, \dots, w_{-r})}_{=q_{-1}} - \beta \underbrace{q\tupleo{(r-1)}}_{=q_0}=:z_0 \in \A \,,
\end{equation}
and thus obtain the digit set conversion from $\B$ to $\A$ in base $\beta$ as $(r+1)$-local function, with anticipation~$0$ and memory~$r$. The requirement that $q(0, \ldots, 0) = 0$, i.e., zero output of the weight function~$q$ for the input of~$r$ zeros, guarantees that $\phi(0, \dots, 0) = 0$. Thus, the first condition of Definition~\ref{def:digitSetConversion} is satisfied. The second condition follows from the equation~\eqref{eq:valuePreserving}.\\

Let us recall the principle of the digit set conversion algorithms based on the rewriting rule \mbox{$x-\beta$}. Assume existence of the weight coefficients set~$\Q$ and the weight function $q: \B^{r} \rightarrow \Q$ for the given numeration system $(\beta, \A)$ and the input digit set~$\B$. To convert $w_{n'} \ldots w_1 w_0 \bullet = (w)_{\beta, \B}$ into $z_{n} \ldots z_1 z_0 \bullet = (w)_{\beta, \A}$, first assign the weight coefficients $q \in \Q$ for each position -- independently, so all at once (in parallel). Then multiply the rewriting rule by the weight coefficients~$q$, and add them digit-wise to the input sequence. In fact, it means that the equation~\eqref{eq:conversionFormula} is applied on each position independently. The digit set conversion is computable in parallel thanks to the fact that the weight coefficients are determined as outputs of the weight function~$q$ of a fixed length~$r$.\\

In order to enable this way of computation, we introduce the so-called \emph{Extending Window Method}. It works in two phases, for a given numeration system $(\beta, \A)$ and an input digit set $\B$. First, it finds some weight coefficients set~$\Q \subset \Zbeta$, as described in Definition~\ref{def:weightCoefficientsSet}. This set~$\Q$ then serves as the starting point for the second phase, in which we gradually increment the expected length~$r$, until the weight function~$q: \B^r \rightarrow \Q$ is uniquely defined for each $\tupleo{(r-1)} \in \B^{r}$, as required in Definition~\ref{def:weightFunction}. If both these phases are successful, the local conversion function is finally determined -- we use the weight function outputs~$q$ as the weight coefficients in the formula~\eqref{eq:localConversion}.\\

Further in this section, we describe construction of the weight coefficients sets and the weight functions, in the two phases respectively. Convergence of both phases is then discussed in Section~\ref{sec:EWM-convergence}.

\subsection{Phase 1 -- Weight Coefficients Set}\label{sub-sec:EWM-description-phase1}

The goal of the first phase is to compute a weight coefficients set $\Q \subset \Zbeta$, i.e., find a set $\Q \ni 0$ such that
$$ \B + \Q \subset \A + \beta \Q\,. $$
We build a sequence $\Q_0, \Q_1, \Q_2, \dots$ of sets $\Q_k \subset \Zbeta$ iteratively, by extending $\Q_k$ to~$\Q_{k+1}$ in such a way that all elements of the set $\B + \Q_k$ get covered by elements of the extended set~$\Q_{k+1}$:
$$ \B+ \Q_k \subset \A + \beta \Q_{k+1}\,.$$
This procedure is repeated until the extended weight coefficients set $\Q_{k+1}$ is the same as the previous set~$\Q_{k}$.
In the sequel, we use the expression \emph{weight coefficient $q$ covers an element $x$}, meaning that there is a digit $a \in \A$ such that $x = a + \beta q$.\\

We start with $\Q_0 = \{ 0 \}$, search for all weight coefficients~$q_0$ necessary to cover all elements $x \in \B$, and add them to the set~$\Q_0$ to obtain the set~$\Q_1$. Then, assume that we have the set~$\Q_k$ for some $k\geq 1$. The weight coefficients in~$\Q_k$ now may appear as the carries~$q_{-1}$. If there are no suitable coefficients~$q_0$ in the set~$\Q_k$ to cover all sums $x = b + q_{-1}$ of coefficients $q_{-1} \in \Q_k$ and digits $b \in \B$, we extend~$\Q_k$ to~$\Q_{k+1}$ with such suitable coefficients, and increase $k := k+1$. And so on, until there is no need to add more elements into~$\Q_k$, as the set~$\Q_k$ already covers all elements from $\B + \Q_k$, so in fact $Q_k = Q_{k+1}$. Then the weight coefficients set $\Q := \Q_k = Q_{k+1}$ satisfies the Definition~\ref{def:weightCoefficientsSet}. Algorithmic description of this process in quasi-code is summarized in Algorithm~\ref{alg:weightCoefSet}. Section~\ref{sub-sec:EWM-convergence-phase1} discusses the convergence of Phase~1, namely the conditions under which it actually happens that $\Q_{k+1} = \Q_k$ for some $k \in \NN$.\\

\begin{algorithm}
\caption{Search for weight coefficients set~$\Q$ (Phase~1)}
\label{alg:weightCoefSet}
    \begin{algorithmic}[1]
   		\REQUIRE{numeration system $(\beta, \A)$, input digit set~$\B$}
        \STATE $k := -1$
        \STATE $Q_0:=\{0\}$
        \REPEAT
            \STATE $k:=k+1$
        	\STATE set $C_x := \{\frac{x-a}{\beta} \colon a\in\A, x-a \text{ is divisible by }\beta\}$ for each $x \in \B + \Q_{k}$ \label{alg-line:candidates}
            \STATE extend $\Q_k$ to $\Q_{k+1}$ so that $\B + \Q_k \subset \A + \beta \Q_{k+1}$, \\
         	i.e., $C_x \cap \Q_{k+1} \neq \emptyset$ for each $x \in \B + \Q_{k}$ (e.g. by Algorithm~\ref{alg:extendWeightCoefSet})
        \UNTIL{$\Q_k = \Q_{k+1}$}
        \STATE $\Q:=\Q_k$
        \ENSURE{weight coefficients set $\Q$}
    \end{algorithmic}
\end{algorithm}

Note that the extension of~$\Q_k$ to~$\Q_{k+1}$ is not unique. Algorithm~\ref{alg:extendWeightCoefSet} shows two possible ways of such construction. Let $C_x = \{\frac{x-a}{\beta} \colon a\in\A, x-a \text{ is divisible by }\beta\}$ for each $x \in \B + \Q_{k}$, and let $\Q'_{k+1}$ contain all elements of $\Q_k$ and elements from all $C_x$ such that $\#C_x=1$. The set $\Q'_{k+1}$ is then extended to $\Q_{k+1}$ by adding all smallest elements from every $C_x$ such that $C_x \cap \Q'_{k+1} = \emptyset$. Different norms may be used to determine the smallest elements, e.g. the absolute value or the $\beta$-norm from Definition~\ref{def:betaNorm}. Other methods of extending $\Q_k$ to~$\Q_{k+1}$ are suggested in~\cite{JLe-thesis}.

\begin{algorithm}
  \caption{Extending intermediate weight coefficients set $Q_k \rightarrow Q_{k+1}$ (Phase~1)}
  \label{alg:extendWeightCoefSet}
  \begin{algorithmic}[1]
    \REQUIRE{previous interim weight coefficients set~$\Q_{k}$, list of candidates $C_x$ for each $x \in \B + \Q_{k}$}
    \STATE $\Q_{k+1}:=\Q_{k}\cup \{q \colon \# C_x =1, C_x = \{q\}, x \in \B + \Q_{k}\}$
    \FORALL{$x \in \B + \Q_{k}$}
        \IF{$C_x \cap \Q_{k+1}=\emptyset$}
	        \STATE add all smallest elements in absolute value (or, alternatively, in $\beta$-norm) of $C_x$ to $\Q_{k+1}$
        \ENDIF
    \ENDFOR
    \ENSURE{intermediate weight coefficients set $\Q_{k+1}$}
  \end{algorithmic}
\end{algorithm}

\subsection{Phase 2 -- Weight Function}\label{sub-sec:EWM-description-phase2}

In the second phase, we want to find a~length $r \in \NN$ and a~weight function $q: \B^{r} \rightarrow \Q$. We start with the weight coefficients set~$\Q$ obtained in Phase~1. The idea is to reduce the number of necessary weight coefficients which are used to convert a given input digit up to just one single value. This is done by increasing gradually the number~$r$ of considered input digits (to the right). When we know more digits (right neighbours of the processed input digit) that cause the carry from the right, then we may be able to decrease the set of possible carries from the right, and consequently need less weight coefficients to convert the input digit into~$\A$.\\

We introduce a notation for the sets of possible weight coefficients for given input digits. If $w_0 \in \B$, then $\Q_{[w_0]}$ denotes a subset of~$\Q$ such that
$$ (\forall q_{-1} \in \Q)(\exists q_0 \in \Q_{[w_0]})(w_0 + q_{-1} - q_0 \beta \in \A)\,. $$
It means that, since the input digits on the right side from the processed input digit can be arbitrary, any carry~$q_{-1}$ from the set~$\Q$ is possible. However, we may be able to limit the set~$\Q$ to its subset~$\Q_{[w_0]}$ of weight coefficients which allow the conversion of~$w_0$ to~$\A$, due to knowledge of the input digit~$w_0$ itself.\\

By induction with respect to $k \in \NN, k\geq 1$, for all $\tupleo{k}\in \B^{k+1}$, let $\Qwo{k}$ denote a~subset of $\Qwo{(k-1)}$ such that
\begin{equation*}
	 (\forall q_{-1} \in \Qw{1}{k})(\exists q_0 \in \Qwo{k})(w_0 + q_{-1} - q_0 \beta \in \A)\,.
\end{equation*}

The sets $\Qwo{k}$ of possible weight coefficients, and consequently a weight function $q$, are constructed by Algorithm~\ref{alg:weightFunction}. The idea is to check all possible right carries $q_{-1}\in\Q$ and determine a minimal subset of values $q_0 \in \Q$ such that
$$ z_0=w_0 + q_{-1} - q_0 \beta \in \A \,. $$

So we obtain a subset $\Q_{[w_0]} \subset \Q$ of weight coefficients which are necessary to cover the sum of the digit~$w_0$ with any carry $q_{-1} \in \Q$. This is done separately for every $w_0 \in \B$, so that we obtain such a subset $\Q_{[w_0]} \subset \Q$ for all $w_0 \in \B$. Then, assuming that we know the input digit~$w_{-1}$, the set of possible carries from the right is also reduced to~$\Q_{[w_{-1}]}$. Thus we may consider the pair $(w_0, w_{-1})$ of input digits, and reduce the set $\Q_{[w_0]}$ to a set $\Q_{[w_0, w_{-1}]} \subset \Q_{[w_0]}$, which is again a minimal subset necessary to cover all elements of $w_0 + \Q_{[w_{-1}]}$.\\

In the $k$-th step, we search for a minimal subset $\Qwo{k} \subset \Qwo{(k-1)}$ such that
\begin{equation*}
    w_0 + \Qw{1}{k} \subset \A + \beta \Qwo{k}\,.
\end{equation*}
Apart from the processed digit itself, we consider here also $k$~digits on the right. To  construct the set~$\Qwo{k}$, we select from~$\Qwo{(k-1)}$ such weight coefficients which are necessary to cover the sums of digits $w_0 + q_{-1}$, with all possible carries $q_{-1}$ from the set $\Qw{1}{k}$.\\

Proceeding in this manner may lead to a unique weight coefficient~$q_0$ for long enough $r$-tuple
of considered input digits $(w_0, \ldots, w_{-(r-1)})$. If there is $r\in\NN$ such that
$$ \#\Qwo{(r-1)} = 1 \mathrm{\quad for\ all \quad} \tupleo{(r-1)} \in \B^r \,, $$
then the output $q\tupleo{(r-1)}$ is defined as the only element of~$\Qwo{(r-1)}$.
To verify that
$$ z_0 = \phi(w_{0}, \dots, w_{-r}) = w_0 + \underbrace{q\tuple{1}{r}}_{= q_{-1}} - \beta \underbrace{q\tupleo{(r-1)}}_{= q_0} $$
is an element of the alphabet~$\A$, just recall that $q_0 = q\tupleo{(r-1)}$ is the only element of the set~$\Qwo{(r-1)}$,
which was constructed so that
$$ w_0 + \Qw{1}{(r-1)} \subset \A + \beta \Qwo{(r-1)}\,. $$
At the same time, $q_{-1}=q\tupleo{(r-1)}$ is the only element of~$\Qw{1}{r}$, which is a subset of~$\Qw{1}{(r-1)}$.\\

Unfortunately, finiteness of Phase~2 is not guaranteed. Some ways that might reveal non-convergence are discussed in Section~\ref{sub-sec:EWM-convergence-phase2}.\\

\begin{algorithm}
	\caption{Search for weight function $q$ (Phase 2)}
	\label{alg:weightFunction}
    \begin{algorithmic}[1]
		\REQUIRE{numeration system $(\beta, \A)$, input digit set~$\B$, weight coefficients set~$\Q$}
        \FORALL{$w_0 \in \B$}
            \STATE find set $\Q_{[w_0]} \subset \Q$ such that $ w_0 + \Q \subset \A + \beta \Q_{[w_0]} $ (e.g. by Algorithm~\ref{alg:minimalSet})
        \ENDFOR
        \STATE $k:=0$
        \WHILE{$\max\{\#\Qwo{k}\colon \tupleo{k}\in \B^{k+1} \} > 1$}
            \STATE $k:= k +1$
            \FORALL{$\tupleo{k}\in \B^{k+1}$}
                \STATE find set $\Qwo{k} \subset \Qwo{(k-1)}$ such that $w_0 + \Qw{1}{k} \subset \A + \beta \Qwo{k}$\\ (e.g. by Algorithm~\ref{alg:minimalSet})
            \ENDFOR
        \ENDWHILE
        \STATE $r := k+1$
        \STATE $q\tupleo{(r-1)} :=$ the only element of $\Qwo{(r-1)}$, for each $\tupleo{(r-1)} \in \B^{r}$
        \ENSURE weight function $q: \B^{r} \rightarrow \Q$
    \end{algorithmic}
\end{algorithm}

Similarly to Phase~1, the choice of~$\Qwo{k}$ is not unique. A list of different methods of choice is in~\cite{JLe-thesis}, the Algorithm~\ref{alg:minimalSet} below describes just two of them. Given a~$(k+1)$-tuple of input digits $(w_0, \ldots, w_{-k}) \in \B^{k+1}$, the set of possible carries $\Qw{1}{k}$ from the right to~$w_0$, and the previous set of possible weight coefficients $\Qwo{(k-1)}$ for $w_0$, let $D_x= \{q_0\in\Qwo{(k-1)}\colon\, \exists \, a\in\A: x=a+\beta q_0 \}$ for each $x \in w_0 + \Qw{1}{k}$. First, we put into $\Qwo{k}$ the elements of all $D_x$ such that $\#D_x=1$. Then, we repeat the following procedure \textbf{while} the set $D=\{D_x\colon x \in w_0 + \Qw{1}{k}, D_x \cap \Qwo{k} = \emptyset\}$ is non-empty. Let $D'$ be all sets in~$D$ of the minimal size, and let $g$~be the center of gravity of $\Qwo{k}$ considered as complex numbers. Let $T$~be all elements of $\bigcup D'$ closest to~$g$ in absolute value, or alternatively, the smallest in $\beta$-norm (see Definition~\ref{def:betaNorm}). We add a deterministically chosen element of~$T$ to $\Qwo{k}$, update~$D$, and continue the \textbf{while} loop. When the loop ends, i.e., every~$D_x$ has non-empty intersection with $\Qwo{k}$, we have the desired result.\\

\begin{algorithm}
	\caption{Search for set $\Qwo{k} \subset \Qwo{(k-1)}$ (Phase~2)}
    \label{alg:minimalSet}
  	\begin{algorithmic}[1]
    \REQUIRE{input digits $(w_0, \ldots, w_{-k}) \in \B^{k+1}$, set of possible carries $\Qw{1}{k}$, previous set of possible weight coefficients $\Qwo{(k-1)}$}
    \STATE $D_x:= \{q_0\in\Qwo{(k-1)}\colon\, \exists \, a\in\A: x=a+\beta q_0 \}$ for each $x \in w_0 + \Qw{1}{k}$
    \STATE $D:=\{D_x \colon x \in w_0 + \Qw{1}{k}\}$
    \STATE $\Qwo{k}:=\{q\colon \{q\}\in D\}$
	\STATE $D:=\{D_x\in D \colon D_x \cap \Qwo{k} = \emptyset\}$
    \WHILE{$D\neq\emptyset$}
        \STATE $m:=\min\{\#D_x \colon D_x\in D \}$
	   	\STATE $D':=\{D_x \in D \colon \#D_x=m\}$
		\STATE $g:=$ center of gravity of elements of $\Qwo{k}$ as complex numbers
		\STATE $T:=$ elements of $\bigcup D'$ which are closest to $g$ in absolute value \\
			     (alternatively, $T:=$ elements of $\bigcup D'$ which are smallest in $\beta$-norm)
  		\STATE $q:=$ deterministically chosen element of $T$ (e.g. the lexicographically smallest one)
        \STATE $\Qwo{k}:= \Qwo{k} \cup \{q\}$
        \STATE $D:=\{D_x\in D \colon D_x \cap \Qwo{k} = \emptyset\}$
    \ENDWHILE
    \ENSURE{$\Qwo{k}$}
  \end{algorithmic}
\end{algorithm}

Notice that, for a given length $r \in \NN$, the number of calls of Algorithm~\ref{alg:minimalSet} within Algorithm~\ref{alg:weightFunction} is
$$ \sum_{k=0}^{r-1}  \#\B^{k+1} = \#\B \frac{\#\B^r-1}{\#\B-1}\,. $$
It implies that the time complexity grows exponentially. The required memory is also exponential, as we have to store the sets~$\Qwo{k}$ for all $(w_0, \dots, w_{-k}) \in \B^{k+1}$, at least for $k = r-1$.\\

\section{Convergence of Extending Window Method}\label{sec:EWM-convergence}

\subsection{Convergence of Phase 1}\label{sub-sec:EWM-convergence-phase1}

In this section, we show that, if the Extending Window Method converges, then the base $\beta \in \CC$ must be \emph{expanding}, i.e., all its conjugates are greater than~$1$ in modulus. Then we prove that this property -- expanding base -- is also a sufficient condition for convergence of Phase~1, provided that the alphabet $\A \subset \Zbeta$ contains at least one representative of each congruence class modulo~$\beta$ in~$\Zbeta$. We see from Theorem~\ref{thm:modulo-beta&beta-1} that the requirements put on the alphabet~$\A$ are in line with the necessary conditions on~$(\beta, \A)$ for parallel addition.

\begin{theorem}\label{thm:betaMustBeExpanding}
Let $\A \subset \Zbeta$ be an alphabet such that $1 \in \A[\beta]$. If the Extending Window Method with the rewriting rule $x - \beta$ converges for numeration system~$(\beta, \A)$, then the base~$\beta$ is expanding.
\end{theorem}
\begin{proof}
By Corollary 3.6 in~\cite{JLe-article}, if the numeration system $(\beta, \A)$ allows parallel addition without anticipation (i.e., $t=0$ in Definition~\ref{def:plocal}),	then $\beta$~is expanding. Parallel addition produced by the~EWM is indeed without anticipation, since there is no carry to the right when using the rewriting rule $x-\beta$.
\end{proof}

We need to define a norm in~$\Zbeta$, in order to prove Lemma~\ref{lem:suffCondPhase1}, which provides a finite set of weight coefficients~$\Q$. Finiteness of~$\Q$ is crucial for the proof of convergence of Phase~1. We exploit the fact that $\Zbeta=\{\sum_{i=0}^{d-1} u_i \beta^i\colon u_i\in\ZZ\}$, where $d$~is the degree of~$\beta$, if and only if $\beta$~is an algebraic integer. Hence, there is an obvious bijection $\pi:\Zbeta \rightarrow \ZZ^{d}$ given by
\begin{equation*}
    \pi(u)=\vect{u} \quad \text{ for every } u=\sum_{i=0}^{d-1} u_i \beta^i \in \Zbeta \,.
\end{equation*}
Using the concept of companion matrix, the additive group~$\ZZ^d$ can be equipped with multiplication such that the mapping~$\pi$ is a ring isomorphism (e.g. see~\cite{katai}). For our purpose, the following lemma is sufficient.

\begin{lemma}\label{lem:multiplyBeta}
Let $\beta \in \CC$ be an algebraic integer with the minimal polynomial $m_\beta (x) = x^d + p_{d-1} x^{d-1} + \cdots + p_1 x + p_0 \in \ZZ[x]$. If $S_\beta$~is the companion matrix of~$m_\beta$, i.e.,
\begin{equation*}
	S_\beta = \begin{pmatrix}
	0 & 0 & \cdots & 0 & -p_0 \\
	1 & 0 & \cdots & 0 & -p_1 \\
	0 & 1 & \cdots & 0 & -p_2 \\
	\vdots &   & \ddots & & \vdots \\
	0 & 0 & \cdots & 1 & -p_{d-1}
    \end{pmatrix} \in \ZZ^{d\times d} \, ,
\end{equation*}
then $\pi(\beta u) = S_\beta \cdot \pi(u)$ for any $u= \sum_{i=0}^{d-1} u_i \beta^i\in \Zbeta$.
\end{lemma}
\begin{proof}
If $\pi(u)=\vect{u}$, then
\begin{align*}
	\pi(\beta u)   & =\pi\left(\beta \sum_{i=0}^{d-1} u_i \beta^i \right) = \pi\left(u_{d-1}\underbrace{(-p_{d-1}\beta^{d-1}- \dots -p_1\beta-p_0)}_{=\beta^d} +\sum_{i=0}^{d-2} u_i \beta^{i+1}\right) \\
	               & =\pi\left( -p_0 u_{d-1} + \sum_{i=1}^{d-1} (u_{i-1}- u_{d-1} p_i) \beta^i \right) \\
	               & =( -p_0 u_{d-1}, u_{0} - u_{d-1} p_1, \dots, u_{d-2} - u_{d-1} p_{d-1} ) = S_\beta \cdot \pi(u)\,.
\end{align*}
\end{proof}

We define a vector norm and a matrix norm induced by a given diagonalizable matrix, and the following Lemma~\ref{lem:propertiesSbeta} shows selected properties of the norm given by the companion matrix $S_\beta$ and\ $S_\beta^{-1}$.

\begin{definition}\label{def:newNorm}
Let $M \in \CC^{n\times n}$ be a diagonalizable matrix, and let $P \in \CC^{n\times n}$ be a~nonsingular matrix which diagonalizes~$M$, i.e., $M = P^{-1}DP$ for some diagonal matrix $D \in \CC^{n\times n}$. We define a~\emph{vector norm} $\norm{\cdot}{M}$ by
\begin{equation*}
    \norm{x}{M} := \norm{Px}{2} \quad \mathrm{for \ all} \ x \in \CC^n \, ,
\end{equation*}
where $\norm{\cdot}{2}$ is the Euclidean norm. A~\emph{matrix norm} $\Mnorm{\cdot}{M}$ is induced by the vector norm $\norm{\cdot}{M}$ by
\begin{equation*}
	\Mnorm{A}{M} := \sup_{\norm{x}{M} = 1} \norm{Ax}{M} \quad \mathrm{for \ all} \ A \in \CC^{n\times n} \, .
\end{equation*}
\end{definition}

\begin{lemma}\label{lem:propertiesSbeta}
Let $\beta$ be an algebraic integer of degree~$d$. If $S_\beta$ is the companion matrix of the minimal monic polynomial $m_\beta$ of $\beta$, then
\begin{align*}
   \Mnorm{S_\beta}{S_\beta}         & = \max \{|\beta'| \colon \beta' \text{ is conjugate of } \beta\},\text{ and } \\
   \Mnorm{S_\beta^{-1}}{S_\beta}    & = \max \left\{\frac{1}{|\beta'|} \colon \beta' \text{ is a conjugate of } \beta\right\} \, .
   \end{align*}
\end{lemma}
\begin{proof}
It is well known that the characteristic polynomial of the companion matrix~$S_\beta$ is~$m_\beta$ (see e.g.~\cite{horn}). Since the minimal polynomial~$m_\beta$ has no multiple roots, $S_\beta$~is diagonalizable over~$\CC$. Namely, there is a~nonsingular complex matrix~$P$ such that $S_\beta=P^{-1}DP$, where $D$~is diagonal matrix with the conjugates of~$\beta$ on the diagonal. Therefore, the norms $\norm{\cdot}{S_\beta}$ and $\Mnorm{\cdot}{S_\beta}$ are well-defined. Since the matrix~$S_\beta^{-1}$ is also diagonalized by~$P$, the vector norms $\norm{\cdot}{S_\beta}$ and~$\norm{\cdot}{S_\beta^{-1}}$ are the same, and so are the induced matrix norms $\Mnorm{\cdot}{S_\beta}$ and~$\Mnorm{\cdot}{S_\beta^{-1}}$.

Now, we use a known result from matrix theory~\cite{horn}: if $M \in \CC^{n\times n}$ is a~diagonalizable matrix, then the spectral radius $\rho(M)$ of the matrix~$M$ equals $\rho(M) = \Mnorm{M}{M}$. Since the eigenvalues of~$S_\beta$ are the conjugates of~$\beta$, we have
\begin{equation*}
	\Mnorm{S_\beta}{S_\beta} = \rho(S_\beta) = \max \{|\beta'| \colon \beta' \text{ is conjugate of } \beta\}\,.
\end{equation*}
Similarly,
\begin{equation*}
	\Mnorm{S_\beta^{-1}}{S_\beta} = \Mnorm{S_\beta^{-1}}{S_\beta^{-1}} = \rho(S_\beta^{-1})= \max \{\frac{1}{|\beta'|} \colon \beta' \text{ is conjugate of } \beta\}\,,
\end{equation*}
where we use the fact that the eigenvalues of~$S_\beta^{-1}$ are reciprocal of the eigenvalues of~$S_\beta$.
\end{proof}

Finally, we may define a norm in $\Zbeta$.

\begin{definition}\label{def:betaNorm}
Let $\pi$ be the isomorphism between $\Zbeta$ and $(\ZZ^d,+,\odot_\beta)$. Using notation from the previous Lemma~\ref{lem:propertiesSbeta}, we define \emph{$\beta$-norm}  $\normBeta{\cdot}: \Zbeta \rightarrow \RR^+_0$ by
$$ \normBeta{x} = \norm{\pi(x)}{S_\beta} \quad \mathrm{for \ all} \ x \in \Zbeta \, . $$
\end{definition}

An important property of the $\beta$-norm is that, for a given constant $K > 0$, there are only finitely many elements of~$\Zbeta$ bounded by~$K$ in this norm. The explanation is as follows: images of elements of~$\Zbeta$ under the isomorphism $\pi$ are integer vectors, and there are only finitely many integer vectors in any finite-dimensional vector space bounded by any norm. It is a consequence of equivalence of all norms on a~finite-dimensional vector space.

\begin{lemma}\label{lem:suffCondPhase1}
Let $\beta \in \CC$ be an expanding algebraic integer of degree~$d$. If $\A$~and~$\B$ are finite subsets of~$\Zbeta$ such that $\A$~contains at least one representative of each congruence class modulo~$\beta$ in~$\Zbeta$, then there exists a finite set $\Q \subset \Zbeta$ such that $ \B + \Q \subset \A + \beta \Q$.
\end{lemma}
\begin{proof}
We use the mapping $\pi: \Zbeta \rightarrow \ZZ^{d}$ and the $\beta$-norm $\normBeta{\cdot}$ to give a~bound on elements of~$\Zbeta$. Let $\gamma$~be the smallest conjugate of~$\beta$ in modulus. Denote $C := \max\{ \normBeta{b-a} \colon a \in \A, b \in \B \}$. Consequently, set
\begin{equation}\label{eq:R+Q}
    R := \frac{C}{|\gamma| - 1} \mathrm{\quad and \quad} \Q := \{q \in \Zbeta \colon \normBeta{q} \leq R\} \,.
\end{equation}
By Lemma~\ref{lem:propertiesSbeta}, we have
$$ \Mnorm{S_\beta^{-1}}{S_\beta} = \max \left\{\frac{1}{|\beta'|} \colon \beta' \text{ is conjugate of } \beta \right\} = \frac{1}{|\gamma|}\,. $$
Also, we have $|\gamma| > 1$, as $\beta$~is an expanding algebraic integer. Since $C > 0$, the set~$\Q$ is non-empty. Any element $x = b + q \in \Zbeta$ with $b \in \B$
and $q \in \Q$ can be written as $x = a + \beta q'$ for some $a \in \A$ and $q' \in \Zbeta$,
due to the presence of at least one representative of each congruence class modulo~$\beta$ in~$\A$. Using the isomorphism~$\pi$ and Lemma~\ref{lem:multiplyBeta}, we may write $\pi(q') = S^{-1}_\beta \cdot \pi(b-a+q)$. We prove that $q'$ is in~$\Q$:
\begin{align*}
    \normBeta{q'}   & = \norm{ \pi \left(q'\right)}{S_\beta} = \norm{S^{-1}_\beta \cdot \pi \left(b - a + q \right)}{S_\beta} \leq \Mnorm{S^{-1}_\beta}{S_\beta} \normBeta{b - a + q} \\
                    & \leq \frac{1}{|\gamma|} \left( \normBeta{b-a} + \normBeta{q} \right) \leq \frac{1}{|\gamma|} \left( C + R \right) = \frac{C}{|\gamma|} \left( 1 + \frac{1}{|\gamma|-1} \right) = R \,.
\end{align*}
Hence $q' \in \Q$, and thus $x = b + q \in \A + \beta \Q$. Since there are only finitely many elements of~$\ZZ^{d}$ bounded by the constant~$R$, the set~$\Q$ must be finite.
\end{proof}

The way how candidates for the weight coefficients are chosen at line~\ref{alg-line:candidates} in Algorithm~\ref{alg:weightCoefSet} is the same as in the proof of Lemma~\ref{lem:suffCondPhase1}. Therefore, the convergence of Phase~1 is guaranteed by the following theorem.

\begin{theorem}\label{thm:suffCondPhase1}
Let $\beta \in \CC$ be an algebraic integer. Let $\A \subset \Zbeta$ be an alphabet containing at least one representative of each congruence class modulo~$\beta$ in~$\Zbeta$, and let $\B \subset \Zbeta$ be an input alphabet. If $\beta$~is expanding, then Phase~1 of the Extending Window Method converges.
\end{theorem}
\begin{proof}
Let $R > 0$ be the constant and $\Q \subset \Zbeta$ the finite set from \eqref{eq:R+Q} in Lemma~\ref{lem:suffCondPhase1}, for the alphabet~$\A$ and the input alphabet~$\B$. We prove by induction that all intermediate weight coefficient sets~$\Q_k$ in Algorithm~\ref{alg:weightCoefSet} are subsets of the finite set~$\Q$. Let us start with $\Q_0 = \{0\}$, whose elements are bounded by any positive constant. Suppose that the intermediate weight coefficients set~$\Q_k$ has elements bounded by the constant~$R$. We see from the proof of Lemma~\ref{lem:suffCondPhase1} that the candidates for the set~$\Q_{k+1}$ obtained in~$C_x$ at line~\ref{alg-line:candidates} of Algorithm~\ref{alg:weightCoefSet} are also bounded by~$R$. Thus, the next intermediate weight coefficients set~$\Q_{k+1}$ has elements bounded by the constant~$R$ as well, i.e., $\Q_{k+1} \subset \Q$. Since $\#\Q$ is finite and $\Q_0 \subsetneq \Q_1 \subsetneq \Q_2 \subsetneq \cdots \subset \Q$, the Phase~1 successfully ends once we obtain $\Q_k = \Q_{k+1}$.
\end{proof}

\subsection{Convergence of Phase 2}\label{sub-sec:EWM-convergence-phase2}

We do not have any straightforward conditions, sufficient or necessary, for convergence of Phase~2 of the EWM, based on properties of the base~$\beta$ or alphabet~$\A$. Nevertheless, the non-convergence can be controlled during the course of the algorithm. An easy check of non-convergence can be done by finding the weight coefficient sets~$\Q_{[b, \dots, b]}$ for each $b \in \B$. For that purpose, we introduce a notion of so-called \emph{stable Phase~2},
which is then used also in the main result of this section: the control of non-convergence during Phase~2 is transformed into searching for a cycle in a directed graph.\\

Firstly, we mention several equivalent conditions of non-convergence of Phase~2:

\begin{lemma}\label{lem:equivalentStatementsForNonConvergenePhaseTwo}
The following statements are equivalent for the EWM applied on a numeration system $(\beta, \A)$ and an input alphabet~$\B$:
\begin{enumerate}[i)]
	\item Phase~2 of the EWM does not converge;
	\item $(\forall \, k \in \NN) \, (\exists \, \tupleo{k} \in \B^{k+1}) \, (\#\Qwo{k} \geq 2)$;
	\item $(\exists \, (w_{-j})_{j\geq 0} \in \B^\NN) (\exists \, k_0 \in \NN) (\forall k \geq k_0) (\#\Qwo{k} = \#\Qwo{(k-1)} \geq 2)$.
\end{enumerate}
\end{lemma}
\begin{proof}
\textit{i)}$\iff$\textit{ii):} The {\bf while} loop in Algorithm~\ref{alg:weightFunction} ends if and only if there is $k\in\NN$ such that $\#\Qwo{k}=1$ for all $\tupleo{k} \in \B^{k+1}$.\\
\textit{ii)}$\iff$\textit{iii):} For $\Rightarrow$, there is an infinite sequence $(w_{-j})_{j \geq 0}$ such that $\#\Qwo{k} \geq 2$ for all $k \in \NN$, since $\Qwo{k} \supset \Qwo{(k+1)}$. Hence, the sequence of integers $(\#\Qwo{k})_{k\geq 0}$ is eventually constant. The opposite implication is trivial.
\end{proof}

We need to ensure that the choice of a possible weight coefficients set $\Qwo{k} \subset \Qwo{(k-1)}$ is determined by the input digits~$(w_0, \ldots, w_{-k}) \in \B^{k+1}$ and the set $\Qw{1}{k}$, while the influence of the set $\Qwo{(k-1)}$ is limited. This is formalized in the following definition:

\begin{definition}
Let~$\B$ be an alphabet of input digits. We say that \emph{Phase~2} of the EWM is \emph{stable} if
$$ \Qw{1}{k} = \Qw{1}{(k-1)} \implies \Qwo{k} = \Qwo{(k-1)} $$
for all $k \in \NN, k \geq 2$ and for all $\tupleo{k} \in \B^{k+1}$.
\end{definition}

Although this definition may seem restrictive, it is actually a natural way how an algorithm should be designed. The set~$\Qwo{k}$ is constructed so that
$$ \B + \Qw{1}{k} \subset \A + \beta \Qwo{k} \, , $$
i.e., there is no reason to choose the set $\Qwo{k}$ as a proper subset of $\Qwo{(k-1)}$, as we know that
$$ \B + \underbrace{\Qw{1}{(k-1)}}_{= \Qw{1}{k}} \subset \A + \beta \Qwo{(k-1)} \,, $$
and $\Qw{1}{(k-1)}$ was chosen as sufficient. In other words, if $\Qwo{k}$ is a proper subset of $\Qwo{(k-1)}$, then the set $\Qwo{(k-1)}$ could have been chosen as $\Qwo{k}$ in the previous iteration step already.

We may guarantee that Phase~2 is stable by wrapping any way of choice of the set~$\Qwo{k}$ into a simple {\bf while} loop, see~\cite{JLe-thesis} for details.\\

Now we use the fact that finiteness of Phase~2 implies existence of a length $m \in \NN$ such that, for every~$b \in \B$, the set $\Qb{m}$ contains only one element; where $\Qb{m}$ is a shorter notation for $ \Q_{[\underbrace{\scriptstyle b, \dots, b}_m]} \,. $
The following theorem shows that $\#\Qb{m}$~must decrease every time we increase the length~$m$, otherwise Phase~2 does not converge.

\begin{theorem}\label{thm:bbbCondition}
If $m_0 \in \NN$ and $b\in\B$ are such that the sets $\Qb{m_0}$ and~$\Qb{m_0-1}$ produced by a~stable Phase~2 of the EWM have the same size, then
$$ \#\Qb{m} = \#\Qb{m_0} \qquad \mathrm{\quad for\ every \quad} m \geq m_0-1\,. $$
Particularly, if $\#\Qb{m_0} \geq 2$, then the Phase~2 of the EWM does not converge.
\end{theorem}
\begin{proof}
As $\Qb{m_0} \subset \Qb{m_0-1}$, the assumption of the same size implies $ \Qb{m_0} = \Qb{m_0-1} \,. $ By the assumption that Phase~2 is stable, we have
$$ \Qb{m_0} = \Qb{m_0-1}    \quad \implies \quad    \Qb{m_0+1} = \Qb{m_0}   \quad \implies \quad    \Qb{m_0+2} = \Qb{m_0+1} \quad \implies \quad    \cdots $$
This implies the statement. If $\#\Qb{m_0} \geq 2$, then statement \textit{iii)} in Lemma~\ref{lem:equivalentStatementsForNonConvergenePhaseTwo} holds for the sequence $(b)_{j \geq 0}$.
\end{proof}

We use this result as follows: For all input digits~$b \in \B$, we run Algorithm~\ref{alg:weightFunction} limited only to $k$-tuples $(b, \ldots, b) \in \B^k$. In every iteration~$k$, we check whether $\Qb{k}$ is smaller than $\Qb{k-1}$. If not, then it does not converge for the input $b\cdots b$, and hence the original EWM with input alphabet~$\B$ does not converge either. This check is linear in the length of the window, and thus it is fast.\\

In general, it happens that $\Qwo{k} = \Qwo{(k-1)}$ for some combination of input digits $\tupleo{k}\in\B^{k+1}$, and Phase~2 of the EWM still does converge. Thus, a condition which signifies non-convergence during Phase~2 is more complicated. It can be formulated as searching for an infinite path in a so-called \emph{Rauzy graph}:

\begin{definition}
Let $\B$ be an alphabet of input digits, and let Phase~2 of the EWM be stable.
Consider $k \in \NN, k \geq 2$. We set
\begin{align*}
	V_k &:= \left\{ \tuple{1}{k} \in \B^k \colon \#\Qw{1}{k} = \#\Qw{1}{(k-1)} \right\} \, \mathrm{ and} \\
	E_k &:= \left\{ \tuple{1}{k} \rightarrow \tuple[w']{1}{k} \in V_k \times V_k \colon \vphantom{\tuple[w']{1}{(k-1)}}\right. \\
		& \qquad \qquad \qquad \qquad	\left. \tuple{2}{k}=\tuple[w']{1}{(k-1)}\right\} \,.
\end{align*}
The directed graph \emph{$G_k = (V_k, E_k)$} is called \emph{Rauzy graph of Phase~2 (for the length~$k$)}.
\end{definition}

This term comes from combinatorics on words. The vertices of the Rauzy graph $G_k$ are combinations of input digits for which the size of their possible weight coefficients sets did not decrease with an increment of the length~$k$; whereas in combinatorics on words, the vertices are given as factors of some language. But the directed edges are placed in the same manner -- if some combination of digits without the first digit equals another combination without the last digit.\\

The structure of the Rauzy graph~$G_k$ signifies whether the non-decreasing combinations are such that they cause non-convergence of Phase~2. Existence of an infinite walk in $G_k$ implies that Phase~2 does not converge:

\begin{theorem}\label{thm:infinitePathInRauzyGraph}
Let Phase~2 of the EWM be stable. If there exists $k_0\in\NN, k_0\geq 2$, and $\tupleo{k_0}\in\B^{k_0+1}$ such that
\begin{enumerate}[i)]
	\item $\#\Qwo{(k_0-1)}>1$ and
	\item there is an infinite walk $(\tuple[w^{(i)}]{1}{k_0})_{i \geq 1}$ in $G_{k_0}$ starting in the vertex $\tuple[w^{(1)}]{1}{k_0} = \tuple{1}{k_0} \, ,$
\end{enumerate}
then Phase~2 does not converge.
\end{theorem}
\begin{proof}
Let us set $(w_k)_{k\geq 0} := w_0, w_1^{(1)}, \dots , w^{(1)}_{k_0-1}, w^{(1)}_{k_0}, w_{k_0}^{(2)},w_{k_0}^{(3)},w_{k_0}^{(4)},\dots$; and we prove that
$$\#\Qwo{k} = \#\Qwo{(k_0-1)} > 1 \mathrm{\quad for\ all \quad} k \geq k_0 - 1 \,,$$
so the condition \textit{iii)} in Lemma~\ref{lem:equivalentStatementsForNonConvergenePhaseTwo} is satisfied. Let $\ell \in \NN$. Since $\tuple{(1+\ell)}{(k_0+\ell)}$ is a vertex of~$G_{k_0}$, the set $\Qw{\ell}{(k_0 + \ell)}$ equals $\Qw{\ell}{(k_0 + \ell - 1)}$. As Phase~2 is stable, we have
\begin{align*}
                    & \Qw{\ell}{(k_0 + \ell)} = \Qw{\ell}{(k_0 + \ell-1)} \\
    \implies \quad  & \Qw{(\ell - 1)}{(k_0 + \ell)} = \Qw{(\ell - 1)}{(k_0 + \ell - 1)} \\
                    & \quad \vdots \\
    \implies \quad  & \Qw{1}{(k_0 + \ell)} = \Qw{1}{(k_0 + \ell - 1)}  \\
    \implies \quad  & \Qwo{(k_0 + \ell)} = \Qwo{(k_0 + \ell - 1)} \,.
\end{align*}
Hence, $\#\Qwo{k} = \#\Qwo{(k_0 - 1)} > 1$ for all $k \geq k_0 - 1$.
\end{proof}

Let us remark that existence of an infinite walk in a finite graph is equivalent to existence of a cycle in the graph. Thus, if there is an infinite walk, we may find another one whose sequence of vertices is eventually periodic. This fact can be used for revealing non-convergence of Phase~2 during its run. Namely, in $k$-th iteration, we construct the Rauzy graph $G_{k+1}$ and check whether there is an infinite walk in~$G_k$ starting in $\tuple{1}{k_0}$. This modification of Phase~2 is elaborated in~\cite{JLe-thesis}.

\section{Implementation of Extending Window Method}\label{sec:EWM-implementation}

We implement the Extending Window Method (EWM) in a~more general setting than explained above. Let $\omega$ be an algebraic integer. Parallel addition is searched for a~numeration system $(\beta,\A)$ such that $\beta \in \Zomega$ and $\A \subset \Zomega$. All elements of~$\Zomega$ are algebraic integers, since $\omega$~is an algebraic integer, and clearly $\Zbeta \subset \Zomega$. We run the EWM just as described above, but all steps being now computed in~$\Zomega$, instead of~$\Zbeta$.

Let us remark that, for instance, the necessary condition from Theorem~\ref{thm:modulo-beta&beta-1}, that $\A$~must contain all representatives modulo $\beta-1$, is not valid in~$\Zomega$ anymore, since congruence classes in~$\Zomega$ and~ $\Zbeta$ are different if $\Zbeta \subsetneq \Zomega$ -- see the discussion in conclusion of~\cite{JLe-article}. Nevertheless, the statements in Theorem~\ref{thm:betaMustBeExpanding} ($\beta$~must be expanding), Theorem~\ref{thm:suffCondPhase1} (convergence of Phase~1),
and Theorems~\ref{thm:bbbCondition} and \ref{thm:infinitePathInRauzyGraph} (control of convergence of Phase~2) can be proven also for~$\Zomega$, see~\cite{JLe-thesis}.\\

Our implementation of the EWM in SageMath can be downloaded from~\cite{githubParAdd}. User information is provided in \verb+readme.md+, and more details about the implementation can be found also in~\cite{JLe-thesis}. We remark that the program allows to choose an algebraic integer~$\omega$ by means of its minimal polynomial~$m_{\omega}$, then the base $\beta \in \Zomega$ and the alphabet $\A \subset \Zomega$. An input alphabet~$\B \subsetneq \A+\A$ or a block length $k \in \NN$ for $k$-block approach may be specified.

One can also select various methods of choice in Phase~1 and in Phase~2. The Algorithm~\ref{alg:extendWeightCoefSet} above describes methods encoded by \verb+'1b'+ (absolute value) and \verb+'1d'+ ($\beta$-norm) for Phase~1. The ways of choice in the Algorithm~\ref{alg:minimalSet} for Phase~2 can be selected by \verb+'2b'+ (center of gravity) and \verb+'2d'+ ($\beta$-norm). Description of other methods available in the program implementation can be found in~\cite{JLe-thesis} or directly in the source code.\\

Before starting Phase~2, the program checks whether it converges for inputs consisting of repetition of a single digit $(b, \ldots, b)$, and stops if not, according to Theorem~\ref{thm:bbbCondition}.
During the computation of Phase~2, Rauzy graphs are constructed, and the computation stops if the conditions of Theorem~\ref{thm:infinitePathInRauzyGraph} are satisfied. In other words, non-finiteness of the EWM can be revealed by these checks.\\

Results of the program for selected examples from Section~\ref{sec:EWM-results} are available at~\cite{JLe-results}. Various combinations of methods in Phase~1 and Phase~2 are included. Log-files provide information about the course of the computations, and the resulting weight functions~$q$ are stored in CSV format when the EWM is successful.

\section{Results Obtained by Extending Window Method}\label{sec:EWM-results}

This chapter shows selected results of the Extending Window Method (EWM). Firstly, we focus on numeration systems with parallel addition algorithms known from previous works of various authors -- we investigate whether the EWM delivers the same algorithms as they have developed manually. Then, we describe a set of new results -- parallel addition algorithms not known so far.

\subsection{Integer Bases}\label{sub-sec:EWM-results-integer}

Let us start with positive integer bases $\beta = b \in \NN$, $b \geq 2$, which certainly are algebraic integers; and they have no algebraic conjugates other than themselves, so no conjugates of modulus~$\leq 1$. We apply the EWM on (several samples of) numeration systems with these bases, analyzed manually in previous works:

\begin{itemize}

    \item A.~Avizienis -- bases $\beta = b \geq 3$, with (symmetric) alphabet $\A = \{ -a, \ldots, 0, \ldots, a \}$, using the smallest possible $a = \lceil (\beta+1)/2 \rceil$: The EWM provides the same parallel addition algorithms as introduced by A.~Avizienis in~\cite{Avizienis}. This was tested on a sample of numeration systems $(3, \{-2, \ldots, +2\})$, $(4, \{-3, \ldots, +3\})$, $(7, \{-4, \ldots, +4\})$, or $(10, \{-6, \ldots, +6\})$; and the pattern of the resulting parallel addition algorithms shows that the EWM would work analogously and correctly, using any positive integer $\beta \in \NN$, $\beta \geq 3$. The key parameters of the parallel addition algorithms obtained here are as follows:
        \begin{itemize}
            \item weight coefficients $q_j = q_j (w_j) \in \Q = \{0, \pm 1\}$ depend on one position only, and thus
            \item digits of the sum $z_j = z_j(w_j, q_j, q_{j-1}) = z_j(w_j, w_{j-1})$ are $2$-local function, with memory $r = 1$.
        \end{itemize}

    \item C.~Y.~Chow, J.~E.~Robertson -- even bases $\beta = b = 2a$ with $a\in\NN$, $a\geq1$, and with (symmetric) alphabet $\A = \{ -a, \ldots, 0, \ldots, a \}$: Again, the EWM delivers the same parallel addition algorithms as published earlier in~\cite{ChowRobertson} by Chow \& Robertson for the tested sample of numeration systems $(2, \{-1, 0, +1\})$, $(4, \{-2, \ldots, +2\})$ or $(10, \{-5, \ldots, +5\})$, and would do the same with any positive even base $\beta = 2a$. The key parameters of these parallel addition algorithms are as follows:
        \begin{itemize}
            \item weight coefficients $q_j = q_j (w_j, w_{j-1}) \in \Q = \{0, \pm 1\}$ depend on two positions (due to smaller $\#\A$),
            \item digits of the sum $z_j = z_j(w_j, q_j, q_{j-1}) = z_j(w_j, w_{j-1}, w_{j-2})$ are $3$-local function, with memory $r = 2$.
        \end{itemize}

    \item B.~Parhami -- any positive integer base $\beta \geq 2$, with alphabet $\A = \{ -d, \ldots, 0, \ldots, b-d \}$, $0 \leq d \leq b$, of the minimal size $\# \A = \beta+1$ for parallel addition: Algorithms published in~\cite{Parhami,FrPeSv2} for these numeration systems are now obtained equally also by the EWM, for any choice of $d \in \{0, \ldots, b\}$, i.e., for all shapes of alphabets in question (positive or negative or mixed, symmetric or non-symmetric). This was tested explicitly on numeration systems $(2, \{0, 1, 2\})$, $(3, \{-1, \ldots, 2\})$, $(4, \{-3, \ldots, +1\})$, or $(7, \{-4, \ldots, +3\})$, and works in general, with parameters:
        \begin{itemize}
            \item weight coefficients $q_j = q_j (w_j, w_{j-1}) \in \Q = \{0, \pm 1\}$ or $\Q = \{0, 1, 2\}$ or $\Q = \{-2, -1, 0\}$, depending on two positions (due to the minimal alphabet size $\#\A$ for parallel addition),
            \item digits of the sum $z_j = z_j(w_j, q_j, q_{j-1}) = z_j(w_j, w_{j-1}, w_{j-2})$ are $3$-local function, with memory $r = 2$.
        \end{itemize}

    \item Also for negative integer bases $\beta = -b \in \ZZ$, $b \geq 2$, we can apply the EWM, because they are algebraic integers without conjugates of modulus $\leq 1$. Again, we focus on alphabets $\A = \{ -d, \ldots, 0, \ldots, b-d \}$, $0 \leq d \leq b$, of the minimal size $\#\A = b+1$ for parallel addition. Testing was done on numeration systems $(-2, \{0, 1, 2\})$, $(-3, \{-1, \ldots, 2\})$, $(-4, \{-3, \ldots, +1\})$, or $(-7, \{-4, \ldots, +3\})$; and also here the EWM produced the same algorithms as derived manually earlier in~\cite{FrPeSv2}, with parameters:
        \begin{itemize}
            \item weight coefficients $q_j = q_j (w_j, w_{j-1}) \in \Q = \{0, \pm 1\}$ or $\Q = \{0, 1, 2\}$ or $\Q = \{-2, -1, 0\}$, depending on two positions (due to the minimal alphabet size $\#\A$ for parallel addition),
            \item digits of the sum $z_j = z_j(w_j, q_j, q_{j-1}) = z_j(w_j, w_{j-1}, w_{j-2})$ are $3$-local function, with memory $r = 2$.
        \end{itemize}

\end{itemize}

Table~\ref{tab:parallel-addition-algo_integer-bases} summarizes these results, together with the basic parameters of the respective numeration systems and parallel addition algorithms.

\begin{table}[h]
\begin{center}
\begin{tabular}{|c|c|c|c|c|c|}
\hline
    Base                        & Minimal                   & Alphabet  & Alphabet              & Locality      & Same \\
    $\beta$                     & polynomial $m_{\beta}$    & $\A$      & size $\#\A$           & $p = 1 + r$   & alg. as \\
\hline \hline

    \textbf{Avizienis}          & $X-b$     & $\{ -a, \ldots, a \}$     & $\#\A = 2 \lfloor \frac{\beta}{2} \rfloor + 3$ & \multirow{2}{*}{$p=2$} & \multirow{2}{*}{\cite{Avizienis}} \\
    \cline{1-2}
    \multicolumn{2}{|c|}{$\beta = b, b \geq 3, b \in \NN$}  & $a = \lceil \frac{b+1}{2} \rceil$ & non-minimal $\#\A$  &           &  \\
\hline \hline

    \textbf{Chow \& Robertson}  & $X-2a$    & \multirow{2}{*}{$\{ -a, \ldots, a \}$}     & $\#\A = \beta+1$      & \multirow{2}{*}{$p=3$}         & \multirow{2}{*}{\cite{ChowRobertson}}\\
    \cline{1-2}
    \multicolumn{2}{|c|}{$\beta = 2a, a\in\NN$}             &           & minimal $\#\A$        &               &  \\
\hline \hline

    \textbf{Parhami}            & $X-b$     & $\{ -d, \ldots, b-d \}$   & $\#\A = \beta+1$      & \multirow{2}{*}{$p=3$}         & \multirow{2}{*}{\cite{Parhami}} \\
    \cline{1-2}
    \multicolumn{2}{|c|}{$\beta = b, b \geq 2, b\in\NN$}    & $0 \leq d < b$    & minimal $\#\A$ &              &  \\
\hline \hline

    \textbf{negative integer}   & $X+b$     & $\{ -d, \ldots, b-d \}$   & $\#\A = |\beta|+1$    & \multirow{2}{*}{$p=3$}         & \multirow{2}{*}{\cite{FrPeSv2}}\\
    \cline{1-2}
    \multicolumn{2}{|c|}{$\beta = -b, b \geq 2, b\in\NN$}   & $0 \leq d < b$    & minimal $\#\A$ &               &  \\
\hline
\end{tabular}
\caption{($1$-block) parallel addition algorithms obtained by the Extending Window Method for numeration systems with integer bases and alphabets are the same as found earlier by Avizienis, Chow \& Robertson, Parhami, or Frougny \& Pelantov\'{a} \& Svobodov\'{a}.}
\label{tab:parallel-addition-algo_integer-bases}
\end{center}
\end{table}

\subsection{Real Bases}\label{sub-sec:EWM-results-real}

Selected classes of real bases were elaborated earlier in~\cite{FrPeSv2} and~\cite{FrHePeSv}, where the parallel addition algorithms are given, even on alphabets of the minimal size for parallel addition. Some of these bases are not eligible for the EWM, since they do not fulfil the core conditions for its usage:

\begin{itemize}

    \item Rational bases $\beta = \pm a/b$ are algebraic numbers, but not algebraic integers.

    \item Real bases $\beta$ being quadratic Pisot numbers, i.e., roots of polynomials $\beta^2 = a\beta \pm b$, are algebraic integers, but they are not expanding, as their algebraic conjugates are smaller than $1$ in modulus.

\end{itemize}

Next, we study other classes of real bases, which are EWM-eligible, being expanding algebraic integers. In the sequel, we show examples of successful EWM applications for such bases, and compare thus obtained results -- parallel addition algorithms -- with those previously found manually (where available).

\subsubsection{Real Bases -- Roots of Integers}\label{sub-sub-sec:EWM-results-real-roots-of-integers}

We limit this class of bases to real roots of integers in the so-called \emph{minimal form}
$$\beta = \sqrt[\ell]{b} \quad \mathrm{with} \quad \ell, b \in \NN, b \geq 2 \, ,$$
where $\beta$ cannot be written as $\beta = \sqrt[\ell'']{c}$, with $\ell = \ell' \ell''$ and $b = c^{\ell'}$, for any $\ell', \ell'', c \in \NN$, $\ell'>1$.\\

Parallel addition algorithms -- as $1$-block $(2\ell + 1)$-local functions -- were obtained manually in~\cite{FrPeSv2} for such bases with alphabets $\A = \{0, \ldots, b\}$ of the minimal possible size $\# \A = b+1$. Also the automated EWM is able to find ($1$-block) parallel addition algorithms for these numeration systems $(\beta, \A)$, see selected examples in Table~\ref{tab:parallel-addition-algo_real-root-of-integer}. The EWM-results on these examples lead to a~hypothesis that, for base $\beta = \sqrt[\ell]{b}$ and alphabet $\A = \{0, \ldots, b\}$, the size of the weight coefficients set obtained during Phase~1 of the EWM is $\#\Q = 3^{\ell}$, and that the $p$-locality of the parallel addition function resulting after Phase~2 is $p = 2\ell + 1$. That is the same $p$-locality as in the parallel addition algorithms provided for these numeration systems in~\cite{FrPeSv2}, nevertheless, the complexity of the algorithms constructed via EWM is a lot higher than of those proposed manually in~\cite{FrPeSv2}.

\begin{table}[h]
\begin{center}
\begin{tabular}{|c|c|c|c|c|c|}
\hline
    Base        & Minimal                   & Alphabet              & Alphabet      & Weight coefficients   & Locality \\
    $\beta$     & polynomial $m_{\beta}$    & $\A$                  & size $\#\A$   & set size $\#\Q$       & $p = 1 + r$ \\
\hline \hline

    $\beta = \sqrt[2]{2}$   & $X^2 - 2$     & $\{0, \ldots, 2\}$    & $\#\A = 3$        & $\#\Q = 9$        & $p = 5$ \\
    \hline

    $\beta = \sqrt[2]{3}$   & $X^2 - 3$     & $\{0, \ldots, 3\}$    & $\#\A = 4$        & $\#\Q = 9$        & $p = 5$ \\
    \hline

    $\beta = \sqrt[2]{5}$   & $X^2 - 5$     & $\{0, \ldots, 5\}$    & $\#\A = 6$        & $\#\Q = 9$        & $p = 5$ \\
    \hline

    $\beta = \sqrt[2]{13}$  & $X^2 - 13$    & $\{0, \ldots, 13\}$   & $\#\A = 14$       & $\#\Q = 9$        & $p = 5$ \\
    \hline

    $\beta = \sqrt[2]{17}$  & $X^2 - 17$    & $\{0, \ldots, 17\}$   & $\#\A = 18$       & $\#\Q = 9$        & $p = 5$ \\
    \hline

    $\beta = \sqrt[2]{21}$  & $X^2 - 21$    & $\{0, \ldots, 21\}$   & $\#\A = 22$       & $\#\Q = 9$        & $p = 5$ \\
    \hline

    $\beta = \sqrt[3]{2}$   & $X^3 - 2$     & $\{0, \ldots, 2\}$    & $\#\A = 3$        & $\#\Q = 27$       & $p = 7$ \\
    \hline

    $\beta = \sqrt[3]{7}$   & $X^3 - 7$     & $\{0, \ldots, 7\}$    & $\#\A = 8$        & $\#\Q = 27$       & $p = 7$ \\
    \hline

    $\beta = \sqrt[4]{5}$   & $X^4 - 5$     & $\{0, \ldots, 5\}$    & $\#\A = 6$        & $\#\Q = 81$       & $p = 9$ \\
    \hline

    $\beta = \sqrt[5]{6}$   & $X^5 - 6$     & $\{0, \ldots, 6\}$    & $\#\A = 7$        & $\#\Q = 243$      & $p = 11$ \\
    \hline \hline

    $\beta = \sqrt[\ell]{b}$   & $X^{\ell} - b$    & $\{0, \ldots, b\}$    & $\#\A = b + 1$    & hypothesis: & hypothesis: \\
    $b \in \NN, b \geq 2$   &       &           &           & $\#\Q = 3^{\ell}$    & $p = 2\ell + 1$ \\

\hline
\end{tabular}
\caption{($1$-block) parallel addition algorithms obtained by the Extending Window Method for selected bases of the form $\beta = \sqrt[\ell]{b}$, with $\ell, b \in \NN$, and non-negative integer alphabets $\A = \{0, \ldots, b\}$ of the minimal size $\#\A = b+1$ for parallel addition. The full look-up tables containing the weight coefficients $q_j = q_j(w_j, w_{j-1}, \ldots, w_{j-(2\ell - 1)})$, assigned to all necessary $2\ell$-tuples of digits $w_j \in \A + \A$, are published in~\cite{JLe-results}. In several cases, we use a~smaller input set $\B \subsetneq \A+\A$, in order to decrease the size of the look-up tables of weight coefficients.}
\label{tab:parallel-addition-algo_real-root-of-integer}
\end{center}
\end{table}

\subsubsection{Real Quadratic Bases}\label{sub-sub-sec:EWM-results-real-quadratic}

For some real quadratic bases (other than the classes already described above), we consider non-integer alphabets, and apply the EWM on such numeration systems. It is successful -- the EWM does provide parallel addition algorithms for those systems, and even on alphabets of minimal size. See Table~\ref{tab:parallel-addition-algo_real-quadratic} for a selection of such results.

\begin{table}[h]
\begin{center}
\begin{tabular}{|c|c|c|c|c|}
\hline
    Base        & Minimal                                              & Alphabet      & Weight coefficients    & Locality \\
    $\beta$     & polynomial $m_{\beta}$                               & size $\#\A$   & set size $\#\Q$        & $p = 1 + r$ \\
\hline \hline

    $\beta = \frac{1}{2} \sqrt{17} - \frac{9}{2}$  & $x^2 + 9x + 16$   & $\#\A = 26$   & $\#\Q = 17$            & $p = 6$ \\
    \hline

	$\beta = \sqrt{5} - 5$                         & $x^2 + 10x + 20$  & $\#\A = 31$   & $\#\Q = 11$            & $p = 4$ \\
    \hline	

    $\beta = \frac{3}{2} \sqrt{5} - \frac{15}{2}$  & $x^2 + 15x + 45$  & $\#\A = 61$   & $\#\Q = 15$            & $p = 4$ \\
\hline
\end{tabular}
\caption{($1$-block) parallel addition algorithms obtained by the Extending Window Method for selected real quadratic bases, using non-integer alphabets of the minimal size $\#\A = \min \{|m_{\beta}(0)|, |m_{\beta}(1)|\}$ for parallel addition. The full look-up tables containing the weight coefficients $q_j = q_j(w_j, w_{j-1}, \ldots, w_{j-(r-1)})$, assigned to all necessary $r$-tuples of digits $w_j \in \A + \A$, are published in~\cite{JLe-results}.}
\label{tab:parallel-addition-algo_real-quadratic}
\end{center}
\end{table}

\subsection{Complex Bases}\label{sub-sec:EWM-results-complex}

It is in the area of numeration systems with complex bases that the EWM provides the most interesting new results, which would be extremely laborious (if not impossible) to achieve by manual calculation.

\subsubsection{Complex Quadratic Bases}\label{sub-sub-sec:EWM-results-complex-quadratic}

Table~\ref{tab:parallel-addition-algo_complex-quadratic} shows selected results of the EWM for numeration systems with complex quadratic bases with non-integer alphabets. These results were obtained via methods 1d and 2b in our EWM-implementation. All results, including also other methods, can be found in~\cite{JLe-results}.

\begin{table}[h]
\begin{center}
\begin{tabular}{|c|c|c|c|c|}
\hline
    Base                                                & Minimal               & Alphabet      & Weight coefficients   & Locality \\
    $\beta$                                             & polynomial $m_\beta$  & size $\#\A$   & set size $\#\Q$       & $p = 1 + r$ \\
\hline \hline

    $\beta = -\imath \sqrt{11} - 4$                         & $x^2 + 8x + 27$   & $\#\A = 36$   & $\#\Q = 13$           & $p = 8$ \\
    \hline

    $\beta = \imath \sqrt{11} - 4$                          & $x^2 + 8x + 27$   & $\#\A = 36$   & $\#\Q = 13$           & $p = 6$ \\
    \hline

    $\beta = \frac{1}{2} \imath \sqrt{11} - \frac{7}{2}$    & $x^2 + 7x + 15$   & $\#\A = 23$   & $\#\Q = 13$           &  $p = 6$ \\
    \hline

    $\beta = \frac{1}{2} \imath \sqrt{7} - \frac{1}{2}$     & $x^2 + x + 2$     & $\#\A = 4$    & $\#\Q = 29$           &  $p = 9$ \\
    \hline

    $\beta = \imath \sqrt{7} - 4$                           & $x^2 + 8x + 23$   & $\#\A = 32$   & $\#\Q = 10$           &  $p = 6$ \\
    \hline

    $\beta = -\frac{3}{2} \imath \sqrt{3} - \frac{15}{2}$   & $x^2 + 15x + 63$  & $\#\A = 79$   & $\#\Q = 13$           &  $p = 4$ \\
    \hline

    $\beta = -\frac{3}{2} \imath \sqrt{3} - \frac{9}{2}$    & $x^2 + 9x + 27$   & $\#\A = 37$   & $\#\Q = 13$           &  $p = 3$ \\
    \hline

    $\beta = \imath \sqrt{2} - 3$                           & $x^2 + 6x + 11$   & $\#\A = 18$   & $\#\Q = 15$           &  $p = 5$ \\
    \hline

    $\beta = -2\imath - 4$                                  & $x^2 + 8x + 20$   & $\#\A = 29$   & $\#\Q = 11$           &  $p = 3$ \\
    \hline

    $\beta = -3\imath - 3$                                  & $x^2 + 6x + 18$   & $\#\A = 25$   & $\#\Q = 15$           &  $p = 5$ \\
\hline
\end{tabular}
\caption{For the class of numeration systems with complex quadratic bases allowing parallel addition with non-integer alphabets, we provide here a~selection from the Extending Window Method results, where the ($1$-block) parallel addition algorithms were obtained on alphabets of minimal size $\#\A$ for parallel addition.}
\label{tab:parallel-addition-algo_complex-quadratic}
\end{center}
\end{table}

\subsubsection{Complex Bases -- Roots of Integers}\label{sub-sub-sec:EWM-results-complex-roots-of-integers}

Here we focus on bases of the type $\beta = \sqrt[\ell]{-b}$, with $\ell, b \in \NN$, $b \geq 2$.\\

For some of these bases, e.g $\beta = \imath\sqrt{2} = \sqrt[2]{-2}$ or the so-called Knuth base $\beta_K = 2\imath = \sqrt[2]{-4}$, the parallel addition algorithms were already provided in~\cite{FrPeSv2}, using non-negative integer alphabets of the minimal size. Those algorithms are given by quite simple formulas, where $q_j = q_j(w_j, w_{j-2})$ and $z_j = z_j(w_j, q_j, q_{j-2})$, so $z_j = z_j(w_j, w_{j-2}, w_{j-4})$. Such formulas are formally $5$-local functions (with memory $r = 4$), but using only every second position. Therefore, we could actually process all odd positions and all even positions separately, via just $3$-local functions (with memory $r = 2$) in each part. When using the EWM on these examples of bases, we also obtain parallel addition algorithms, even for various integer alphabets of minimal size (so not just for non-negative ones); but their formulas are not so simple, although their $p$-locality parameters are the same, i.e., $p = 5$ with memory $r = 4$. For non-integer alphabets, however, the EWM does not provide any new results:

\begin{itemize}

    \item Base $\beta = \imath \sqrt{2} = \sqrt[2]{-2}$, with minimal polynomial $m_{\beta}(X) = X^2+2$, needs the minimal alphabet size for parallel addition equal to $\#\A = 3 = \max\{2, 3\} = \max\{|m_{\beta}(0)|, |m_{\beta}(1)|\}$. The EWM acts as follows:
            \begin{itemize}
                \item On integer alphabets $\A = \{0, 1, 2\}$ or $\A = \{-1, 0, 1\}$: EWM provides parallel addition algorithms with weight coefficients $q_j = q_j (w_j, w_{j-1}, w_{j-2}, w_{j-3})$, i.e., memory $r = 4$, and consequently a $5$-local addition function $z_j = z_j (w_j, q_j, q_{j-1}) = z_j (w_j, \ldots, w_{j-4})$.
                \item On non-integer alphabets $\A = \{0, 1, 1+\imath\sqrt{2}\}$ or $\A = \{0, 1, -\imath\sqrt{2}\}$: EWM gets cycled in Phase 2, so no new parallel addition algorithm is obtained here.
            \end{itemize}

    \item Knuth base $\beta_K = 2 \imath = \sqrt[2]{-4}$, with minimal polynomial $m_{\beta}(X) = X^2+4$, has the minimal alphabet size for parallel addition $\#\A = 5 = \max\{4, 5\} = \max\{|m_{\beta}(0)|, |m_{\beta}(1)|\}$. The EWM ends up as follows:
        \begin{itemize}
            \item On integer alphabets $\A = \{0, \ldots, 4\}$ or $\A = \{-1, \ldots, 3\}$ or $\A = \{-2, \ldots, 2\}$: EWM results in parallel addition algorithms with memory $r = 4$ due to weight coefficients $q_j = q_j (w_j, w_{j-1}, w_{j-2}, w_{j-3})$, and thus $p = 5$ for $z_j = z_j(w_j, q_j, q_{j-1}) = z_j(w_j, \ldots, w_{j-4})$, i.e., a $5$-local function for the digits of the sum.
            \item On non-integer alphabets, e.g. $\A = \{0,1+2\imath, -1-2\imath, 2-2\imath, -2+2\imath\}$: no parallel addition algorithm obtained by EWM, due to several elements $b \in \A+\A$, for which there is no unique weight coefficient for inputs of the form $(b, \ldots, b)$.
        \end{itemize}

\end{itemize}

For the bases mentioned above, it is not any helpful to use the $k$-block concept -- it does not allow to decrease the alphabet size for parallel addition. The EWM correctly reports for such attempts that there are not enough representatives of congruence classes $\mod (\beta-1)$ in the alphabet.\\

Parameters of the $1$-block parallel addition algorithms for bases $\beta = \imath \sqrt{2}$ and $\beta = 2\imath$ are summarized in Table~\ref{tab:parallel-addition-algo_complex-root-of-integer}, together with another two examples: $\beta = -\imath \sqrt{7}$ and $\beta = -\imath \sqrt{11}$. The latter two are elaborated in two ways -- on alphabets of minimal and non-minimal size, respectively, via $5$-local and $3$-local functions, respectively.

\begin{table}[h]
\begin{center}
\begin{tabular}{|c|c|c|c|c|c|}
\hline
    Base            & Minimal                       & Alphabet      & Weight coefficients   & Locality      & Comment \\
    $\beta$         & polynomial $m_\beta$          & size $\#\A$   & set size $\#\Q$       & $p = 1 + r$   & \\
\hline \hline

	$\beta = -\imath \sqrt{11}$    & $x^2 + 11$    & $\#\A = 13$    & $\#\Q = 9$            & $p = 3$      & non-minimal $\#\A$\\
    \cline{3-6}
	                               &               & $\#\A = 12$    & $\#\Q = 9$            & $p = 5$      & minimal $\#\A$ \\
    \hline \hline
    	
    $\beta = -\imath \sqrt{7}$     & $x^2 + 7$     & $\#\A = 9$     & $\#\Q = 9$             & $p = 3$      & non-minimal $\#\A$\\
    \cline{3-6}
                                   &               & $\#\A = 8$     & $\#\Q = 9$             & $p = 5$      & minimal $\#\A$ \\
    \hline \hline

    $\beta = \imath \sqrt{2}$      & $x^2 + 2$     & $\#\A = 3$     & $\#\Q = 9$             & $p = 5$      & minimal $\#\A$ \\
    \hline \hline

     $\beta = 2\imath$             & $x^2 + 4$     & $\#\A = 5$     & $\#\Q = 9$             & $p = 5$      & minimal $\#\A$ \\
\hline
\end{tabular}
\caption{($1$-block) parallel addition algorithms for complex bases of the type $\beta = \sqrt[\ell]{-b}$, with $\ell, b \in \NN$, allowing parallel addition with integer alphabets. On some of the bases, we illustrate how the higher (than minimal) number of elements in alphabet $\#\A$ helps to decrease the $p$-locality of the parallel addition function.}
\label{tab:parallel-addition-algo_complex-root-of-integer}
\end{center}
\end{table}

\subsubsection{Canonical Number Systems}\label{sub-sub-sec:Canonical-Number-Systems}

In the sequel, we investigate in detail another two bases of the type $\beta = \sqrt[\ell]{-b}$, $b \in \NN$:
\begin{itemize}
    \item Penney base $\beta_P = -1 + \imath = \sqrt[4]{-4}$: studied in Section~\ref{sub-sec:EWM-results-Penney};
    \item Eisenstein base $\beta_E = -1 + \exp{\frac{2\pi\imath}{3}} = \frac{-3+\imath\sqrt{3}}{2} = \sqrt[6]{-27}$: studied in Section~\ref{sub-sec:EWM-results-Eisenstein}.
\end{itemize}

For these bases, the EWM provides a rich set of results -- not only regarding $1$-block parallel addition, but using also the $k$-block concept. Here we exploit the fact that both these bases form Canonical Number Systems (CNS -- as studied in~\cite{Kovacs, Kovacs+Petho}) and a~result from~\cite{FrHePeSv}, summarized in the following Remark:

\begin{remark}\label{rem:CNS_k-block}
An algebraic number~$\beta$ and the alphabet $\C = \{ 0, 1, \ldots, |N(\beta)|-1 \}$, where $N(\beta)$ is the norm of~$\beta$ over~$\QQ$, form a~\emph{Canonical Number System (CNS)}, if any element $X$ of the ring of integers $\ZZ[\beta]$ has a unique representation in the form $X = \sum_{k=0}^n x_k \beta^k$, where $x_k \in \C$. In a CNS, block parallel addition is possible on the alphabet $\A = \{ 0, 1, \ldots, 2|N(\beta)|-2 \}$ or on the alphabet $\A = \{-|N(\beta)|+1, \ldots, 0, \ldots, |N(\beta)|-1 \}$.
\end{remark}

Application of this result on Penney and Eisenstein bases means:
\begin{itemize}
    \item the norm of Penney base is $N(\beta_P) = 2$, so $(\beta_P, \C_P)$ with $\C_P = \{0, 1\}$ is a CNS -- therefore, block parallel addition in base $\beta_P$ is possible on alphabets $\A_P = \{ -1, 0, 1 \}$ or $\A_P = \{ 0, 1, 2 \}$;
    \item the norm of Eisenstein base is $N(\beta_E) = 3$, so $(\beta_E, \C_E)$ with $\C_E = \{ 0, 1, 2\}$ is a CNS -- thus block parallel addition in base $\beta_E$ is possible on alphabets $\A_E = \{ -2, -1, 0, 1, 2\}$ or $\A_E = \{ 0, 1, 2, 3, 4 \}$.
\end{itemize}

But the length~$k$ of blocks for the $k$-block parallel addition is not given in Remark~\ref{rem:CNS_k-block} -- so we test it via the EWM, proceeding simply from $k = 2, 3, \ldots$ upwards. The EWM is successful, and generates the algorithms of block parallel addition for both Penney and Eisenstein numeration systems, with $k = 2$ and $k = 3$, respectively. The alphabets $\A_{P_2}$ and $\A_{E_3}$, of sizes $\A_{P_2} = 3$ and $\A_{E_3} = 5$, in these algorithms are exactly those as predicted by the Remark~\ref{rem:CNS_k-block}; both the symmetric and the non-negative sets of consecutive integers. In both cases, the $k$-block concept helps to decrease the size of alphabets for parallel addition by two digits.

\begin{table}[h]
\begin{center}
\begin{tabular}{|c|c|c|c|c|l|p{6cm}|}
\hline
    Base $\beta$    &         Alphabet        &       & Block     & Locality      &  \\
   Minimal polynomial $m_\beta$          &  $\A$                  & $\#\A$   & length    & $p = 1 + r$   & Comment\\
\hline \hline

    \textbf{Penney}     & $\{ 0, \pm 1, \pm \imath \}$  & $5$    & $k=1$     & $p=7$ & minimal $\#\A$ \\
    \cline{2-6}
    $\beta = 1-\imath$  & $\{ 0, \pm 1 \}$              & \multirow{2}{*}{$3$}    & \multirow{2}{*}{$k=2$}     & \multirow{2}{*}{$p=6$} & \multirow{2}{*}{$\#\A$ proposed in \cite{Kovacs+Petho}} \\
    $m_{\beta}=X^2+2X+2$ & $\{ 0, 1, 2 \}$              &        &           &       & \\
\hline \hline

    \textbf{Eisenstein} &  $\{ 0, \pm 1, \pm \omega, \pm \omega^2 \}$                & $7$   & $k=1$ & $p=4$ & minimal $\#\A$ \\
    \cline{2-6}
    $\beta = 1-\omega, \omega = \exp{\frac{2\pi\imath}{3}}$  & \multirow{2}{*}{$\{ 0, 1, \pm \omega, \omega^2 \}$} & \multirow{2}{*}{$5$}   & \multirow{2}{*}{$k=3$} & \multirow{2}{*}{$p=3$} & \multirow{2}{*}{$\#\A$ proposed in \cite{Kovacs+Petho}} \\
    $m_{\beta}=X^2+3X+3$ & & & & &\\
\hline
\end{tabular}
\caption{Parallel addition algorithms obtained by the Extending Window Method for the Penney and Eisenstein (complex) numeration systems with non-integer alphabets are new results; in both $1$-block and $k$-block variants. Detailed comments on these results are elaborated in Sections~\ref{sub-sec:EWM-results-Penney} and~\ref{sub-sec:EWM-results-Eisenstein}. The full look-up tables containing the weight coefficients $q_j = q_j(w_j, w_{j-1}, \ldots, w_{j-r+1})$, assigned to all necessary $r$-tuples of digits $w_j \in \A + \A$, are published in~\cite{JLe-results}.}
\label{tab:parallel-addition-algo_Penney-Eisenstein}
\end{center}
\end{table}

\subsection{Penney Numeration Systems}\label{sub-sec:EWM-results-Penney}

In 1960's, W.~Penney~\cite{Penney} proposed to use the complex base $\beta_P = -1 + \imath$ with the alphabet $\C_P = \{0, 1\}$. Such numeration system $(-1+\imath, \{0, 1\})$ can represent any complex number $X \in \CC$ as $X = (x_n \ldots x_1 x_0 \bullet x_{-1} \ldots)_{-1+\imath}$, with $x_j \in \{0, 1\}$. The Penney base $\beta_P = -1+\imath$ is an algebraic integer, with minimal polynomial $m_{\beta_P}(X) = X^2 + 2X + 2$. So the minimum size of alphabets $\A_P$ allowing ($1$-block) parallel addition in base $\beta_P$, according to Theorem~\ref{thm:minimal-alphabet-size}, is limited by $ \# \A_P \geq \max \{ |m_{\beta_P}(0)|, |m_{\beta_P-1}(0)| \} = \max \{ 2, 5 \} = 5 \, . $ For $k$-block parallel addition, the alphabet size of~$3$ is suggested in Remark~\ref{rem:CNS_k-block}.

\subsubsection{Penney Base with Integer Alphabet: $1$-block}

We consider $5$-digit integer alphabets, e.g. non-negative $\A = \{0, \dots, 4\}$ or symmetric $\A = \{-2, \dots, 2\}$. The EWM does not work here: although Phase 1 finds successfully the weight coefficient sets $\Q$, there are several digits $b \in \A + \A$ not passing the one-letter-input part of Phase 2 in the EWM algorithm.

But we can obtain the parallel addition algorithms due to the fact that $\beta_P^4 = -4$, and using the results obtained earlier for the negative integer base $\gamma = -4$. The modification of algorithm from base $\gamma = -4$ to $\beta_P = -1+\imath = \sqrt[4]{-4}$, with the same alphabet $\A$, is quite simple. We perform conversion of digits from $w_j \in \A+\A$ to $z_j \in \A$, by means of the same weight coefficients $q_j \in \Q$, as follows:
\begin{eqnarray*}
    \mathrm{for} \ (\gamma, \A): & \qquad z_j := w_j + q_{j-1} - q_j \gamma  = w_j + q_{j-1} + 4 q_j \qquad \ \mathrm{where} \quad q_j = q_j (w_j, w_{j-1}) \, ; \\
    \mathrm{for} \ (\beta_P, \A):  & \qquad z_j := w_j + q_{j-4} - q_j \beta_P^4 = w_j + q_{j-4} + 4 q_j \qquad \mathrm{where} \quad q_j = q_j (w_j, w_{j-4}) \,.
\end{eqnarray*}
While in $(-4, \A)$, we had $z_j = z_j (w_j, q_j, q_{j-1}) = z_j (w_j, w_{j-1}, w_{j-2})$, and so parallel addition was a $3$-local function, now in $(-1+\imath, \A)$ we obtain $z_j = z_j (w_j, q_j, q_{j-4}) = z_j (w_j, w_{j-4}, w_{j-8})$, and thus $9$-local function of parallel addition, with memory $r=8$ and anticipation $t=0$.

For the example of symmetric alphabet $\A = \{0, \pm 1, \pm 2\}$, the weight function $q_j : (\A+\A)^2 \rightarrow \{0, \pm 1\} = \Q$ for parallel addition in $(\beta_P, \A)$ has the following form:
\begin{eqnarray*}
    q_j = +1    & \mathrm{for} & (w_j \leq -3) \quad \mathrm{or} \quad (w_j = -2 \ \ \mathrm{and} \ \ w_{j-4} \geq +2) \\
    q_j = -1    & \mathrm{for} & (w_j \geq +3) \quad \mathrm{or} \quad (w_j = +2 \ \ \mathrm{and} \ \ w_{j-4} \leq -2) \\
    q_j = 0 \ \ &              & \mathrm{otherwise} \, .
\end{eqnarray*}
So here we have example of numeration systems $(\beta_P, \A)$ where parallel addition algorithms exist, but the EWM it not able to find them, although its prerequisites are fulfilled.

\subsubsection{Penney Base with Complex Alphabet: $1$-block}

Using the fact that $\pm \imath \equiv_{-2+\imath} \pm 2$ and the Penney base $\beta_P$ fulfils $\beta_P-1 = -2 + \imath$, we can move from the integer alphabet $\{0, \pm 1, \pm 2\}$ to the complex alphabet $\A_P = \{ 0, \pm 1, \pm \imath \}$, still containing representatives of all $5$~congruence classes modulo $\beta_P - 1$, and, moreover, $\A_P \cdot \A_P = \A_P$, i.e., $\A_P$ is closed under multiplication.\\

Application of EWM on numeration system $(-1 + \imath,\{0, \pm 1, \pm \imath\})$ is successful -- provides an algorithm of parallel addition with memory $r=6$ and anticipation $t=0$, so $7$-local function, using the weight coefficient set of size $\# \Q = 45$, as follows:
\begin{equation*}
    z_j := w_j + q_{j-1} - q_j \beta_P = w_j + q_{j-1} - (\imath - 1) q_j  \, , 
\end{equation*}
where $q_j = q_j (w_j, \ldots, w_{j-5})$. Hence, $z_j = z_j (w_j, q_j, q_{j-1}) = z_j (w_j, \ldots, w_{j-6})$.

Having an alphabet $\A_P$ closed under multiplication can be advantageous for efficiency of various arithmetic operations (e.g. multiplication and division), but for the parallel addition itself, this alphabet requires a very large size of the look-up table describing the weight coefficient function $q_j$. The weight function has six arguments, so the maximum size of the look-up table could be $\# (\A_P+\A_P)^6 = 13^6 = 4\,826\,809$. The algorithm obtained by EWM results in the look-up table for $q_j$ of somewhat smaller size $2\,165\,713$ -- but still, rather disadvantageous.

\subsubsection{Penney Base with Integer Alphabet: $2$-block}

The EWM algorithm is successful also in search for parallel addition algorithm in the Penney base with an integer alphabet, using the $k$-block concept for $k = 2$. In this case, it is sufficient to take a (symmetric) alphabet of three digits only, as suggested in Remark~\ref{rem:CNS_k-block} for the~CNS.\\

In the $2$-block concept, we actually regard the $(\beta_P, \A_{P_2})$-representation of a number $x \in \CC$ as a (half-length) representation in a transformed numeration system $(\widetilde{\beta_P}, \widetilde{\A_{P_2}})$:
\begin{align*}\label{Penney_integer_2-block}
    \beta_P &= \imath-1    &   &\rightarrow &   \widetilde{\beta_P} &= \beta_P^2 = -2\imath \\
    \A_{P_2} &= \{0, \pm 1\}\!\!   &   &\rightarrow &   \widetilde{\A_{P_2}} &= \{ \beta_P a_1 + a_0 \, \vert \, a_j \in \A_{P_2} \} = \{ (\imath-1) a_1 + a_0 \, \vert
    \, a_j \in \{0, \pm 1\} \} \,, \, \# \widetilde{\A_{P_2}} = 9 \\
    x_j &\in \A_{P_2}          &   &\rightarrow &   x &= ( x_{2n+1} x_{2n} \ldots x_1 x_0 \bullet x_{-1} x_{-2} \ldots )_{\beta_P} = \sum_{j=-\infty}^{2n+1} x_j \beta_P^j  \\
    & & & & &=\sum_{l=-\infty}^n (\beta_P x_{2l+1} + x_{2l}) (\beta_P^2)^{l} \\
    \widetilde{x}_l &\in \widetilde{\A_{P_2}}  &   &\rightarrow &   x &= ( \widetilde{x}_{n} \ldots \widetilde{x}_0 \bullet \widetilde{x}_{-1} \ldots)_{\widetilde{\beta_P}} = \sum_{l=-\infty}^n \widetilde{x}_{l} (\widetilde{\beta_P})^{l} \quad \mathrm{with} \quad \widetilde{x}_l = \beta_P x_{2l+1} + x_{2l}\,.
\end{align*}

Application of the EWM algorithm on numeration system $(\widetilde{\beta_P}, \widetilde{\A_{P_2}})$ provides the following result:
\begin{itemize}
    \item weight coefficient function $\widetilde{q}: (\widetilde{\A_{P_2}}+\widetilde{\A_{P_2}})^5 \rightarrow \widetilde{\Q}$, with $\widetilde{q}_j = \widetilde{q}_j (\widetilde{w}_j, \ldots, \widetilde{w}_{j-4}) \in \widetilde{\Q} \subset \Zbeta$, described by $60\,721$ distinct input combinations $(\widetilde{w}_j, \ldots, \widetilde{w}_{j-4})$;
    \item parallel addition $\widetilde{z}_j := \widetilde{w}_j + \widetilde{q}_{j-1} - \widetilde{q}_j \widetilde{\beta_P}$ is then a $6$-local function, with memory $r = 5$ and anticipation $t = 0$, with respect to the numeration system $(\widetilde{\beta_P}, \widetilde{\A_{P_2}})$;
    \item this means, in view of the original numeration system $(\beta_P, \A_{P_2})$, a $12$-local function, with parameters of memory and anticipation equal to $r = 10$, $t = 1$ for even positions, and $r = 11$, $t = 0$ for odd positions.
\end{itemize}

So the alphabet $\A_{P_2}$ for $2$-block parallel addition could be limited to just three elements $\{0, \pm 1\}$, and also the look-up table (of $60\,721$ items) describing the weight coefficient function is smaller than in the $1$-block case. Similar results as on the symmetric alphabet $\A_{P_2} = \{ -1, 0, 1\}$ are obtained by EWM also for the non-negative alphabet $\A_{P_2} = \{ 0, 1, 2\}$; the latter result has the same $p$-locality, but larger look-up table describing the weight coefficient function on $114\,481$ distinct input combinations $(\widetilde{w}_j, \ldots, \widetilde{w}_{j-4})$.

\subsubsection{Penney Base with Complex Alphabet: $2$-block}

Attempts to apply the EWM method on non-integer alphabets $\A$ of size $3$, containing $\{0, 1\}$, were not successful with the $2$-block approach. As the third element of the alphabet $\A$, all of the five remaining elements of $\Zbeta$ on the unit circle were tried: $ -1, \pm \imath, \imath $; but they all have failed, in various steps of the EWM algorithm.

\subsection{Eisenstein Numeration Systems}\label{sub-sec:EWM-results-Eisenstein}

Let us denote by $\omega$ the third root of unity $\omega = \exp{\tfrac{2 \pi \imath}{3}} = \tfrac{-1 + \imath \sqrt{3}}{2}$, an algebraic integer with minimal polynomial $m_\omega (X) = X^2 + X + 1$. Eisenstein base is the complex number $\beta_E = -1 + \omega = \tfrac{-3 + \imath \sqrt{3}}{2}$, an algebraic integer with minimal polynomial $m_{\beta_E}(X) = X^2 + 3X + 3$. This number $\beta_E$ generates the set $\ZZ [\beta_E]$ of so-called \emph{Eisenstein integers} of the form:
\begin{equation}\label{Eisenstein_integers}
    \ZZ [\beta_E] = \left\{ \left. \sum_{k=0}^n a_k \beta_E^k \, \right| \, n \in \ZZ^{+}, a_k \in \ZZ \right\} = \left\{ a_1 \beta_E + a_0 \, | \, a_k \in \ZZ \right\} \quad \mathrm{with} \quad \beta_E = \tfrac{-3 + \imath \sqrt{3}}{2} \, .
\end{equation}

Again, we derive the minimum size of alphabets~$\A_E$ (integer or complex) allowing $1$-block parallel addition in this base, due to Theorem~\ref{thm:minimal-alphabet-size}, as $ \# \A_E \geq \max \{ |m_{\beta_E}(0)|, |m_{\beta_E-1}(0)| \} = \max \{ 3, 7 \} = 7 \, . $ For $k$-block parallel addition, the alphabet size of~$5$ is suggested in Remark~\ref{rem:CNS_k-block}.

\subsubsection{Eisenstein Base with Integer Alphabet: $1$-block}

We consider $7$-digit integer alphabets, non-negative $\A = \{0, \ldots, 6\}$ or symmetric $\A = \{-3, \ldots, 3\}$, and apply the~EWM on such numeration systems -- but unsuccessfully. Phase~1 does find the weight coefficient sets~$\Q$, but there are several digits $b \in \A + \A$ not passing the one-letter-input part of Phase~2 in the~EWM.

Parallel addition algorithms for the Eisenstein base $\beta_E$ with the mentioned integer alphabets~$\A$ could be found manually, using the rewriting rule $X^2 + 3X + 3$, but with quite complex weight function; and the resulting parallel addition function would be $9$-local.

\subsubsection{Eisenstein Base with Complex Alphabet: $1$-block}

We have applied the EWM successfully on the following alphabet $\A_E \subset \ZZ [\beta_E]$ of the minimal size:
\begin{equation}\label{Eisenstein_alphabet}
    \A_E = \{ 0, \pm 1, \pm \omega, \pm \omega^2 \} = \{ 0, \pm 1, \pm \omega, \pm (\omega + 1) \} \qquad \text{with} \qquad \# \A_E = 7 \, .
\end{equation}

This alphabet~$\A_E$ is non-integer, but it has many advantages: besides fitting into the minimal $7$-digit size for parallel addition, it is centrally symmetric ($-\A_E = \A_E$), and also closed under multiplication ($\A_E \cdot \A_E = \A_E$).

The EWM provides parallel addition algorithm in numeration system $(\beta_E, \A_E)$, as follows:
\begin{itemize}
    \item the weight coefficients set is, by coincidence, equal to $\Q_E = \A_E + \A_E$ of size $\# \Q_E = 19$;
    \item the memory parameter $r = 3$, i.e., each weight coefficient $q_j \in \Q_E$ depends on (at most) $3$~digits of the representation of $w = w_n \ldots w_1 w_0 \decdot w_{-1} \ldots w_{-m} \in (\A_E + \A_E)^*$:
        \begin{equation*}
            q_j = q_j (w_j, w_{j-1}, w_{j-2}) \in \Q_E \qquad \text{such \ that} \qquad z_j = w_j + q_{j-1} - q_j \beta_E \in \A_E \, ;
        \end{equation*}
    \item thus $p = r+1 = 4$, i.e., the parallel addition is a $4$-local function:
        \begin{equation*}
            z_j = z_j (w_j, q_{j-1} (w_{j-1}, w_{j-2}, w_{j-3}), q_j (w_j, w_{j-1}, w_{j-2})) = z_j (w_j, \ldots, w_{j-3}) \in \A_E \, .
        \end{equation*}
\end{itemize}

The memory $r = 3$ means that description of the weight coefficient function $q_j: (\A_E + \A_E)^3 \rightarrow \Q_E$ requires to provide the values $q_j = q_j (w_j, w_{j-1}, w_{j-2})$ for up to $19^3 = 6\,859$ combinations of triplets $(w_j, w_{j-1}, w_{j-2}) \in (\A_E + \A_E)^3$. This look-up table can be economized by making use of the $6$-fold rotation symmetry of the sets $\A_E$, $\A_E + \A_E$, $\Q_E$. Let us denote:
\begin{equation*}
    \R := \left\{ \left. \pm \omega^k \, \right| \, k \in \ZZ \right\} = \left\{ \pm 1, \pm \omega, \pm \omega^2 \right\} \, .
\end{equation*}

Then, for any element $\rho \in \R$, we have $\rho \A_E = \A_E$, $\rho (\A_E + \A_E) = (\A_E + \A_E)$, $\rho \Q_E = \Q_E$, and, consequently, for any set of digits $w_j, w_{j-1}, w_{j-2}, w_{j-3} \in \A_E + \A_E$:
\begin{align}\label{q_j_rotated}
    q_j (\rho w_j, \rho w_{j-1}, \rho w_{j-2}) & =  \rho q_j (w_j, w_{j-1}, w_{j-2}) \, , \nonumber \\
    \nonumber
    z_j (\rho w_j, \rho w_{j-1}, \rho w_{j-2}, \rho w_{j-3}) & =  \rho w_j + q_{j-1} (\rho w_{j-1}, \rho w_{j-2}, \rho w_{j-3}) - q_j (\rho w_j, \rho w_{j-1}, \rho w_{j-2}) \beta_E  \\
     & =  \rho w_j + \rho q_{j-1} (w_{j-1}, w_{j-2}, w_{j-3}) - \rho q_j (w_j, w_{j-1}, w_{j-2}) \beta_E  \\
     \nonumber
     & =  \rho (w_j + q_{j-1} (w_{j-1}, w_{j-2}, w_{j-3}) - q_j (w_j, w_{j-1}, w_{j-2}) \beta_E)  \\
     \nonumber
     & =  \rho z_j (w_j, w_{j-1}, w_{j-2}, w_{j-3}) \, .
\end{align}


\begin{figure}[h]
    \centering
    \includegraphics[scale=0.5]{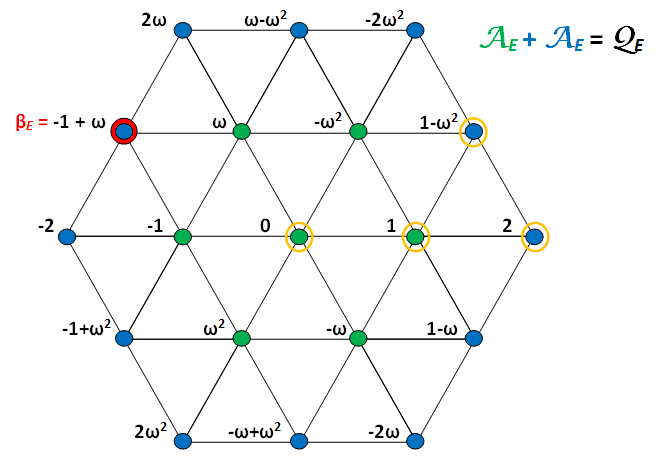}
    \caption{To describe the algorithm of parallel addition in the complex Eisenstein numeration system $(\beta_E, \A_E)$, it is sufficient to provide the weight coefficients $q_j = q_j (w_j, w_{j-1}, w_{j-2})$ for just those triplets $(w_j, w_{j-1}, w_{j-2}) \in (\A_E + \A_E)^3$ starting with $w_j \in \{ 0, 1, 2, 1 - \omega^2 \}$. The remaining weight coefficients are then obtained by using the equation \eqref{q_j_rotated} for all $\rho \in \R$.}
    \label{Eisenstein_complex_1-block_picture}
\end{figure}


\subsubsection{Eisenstein Base with Integer Alphabet: $3$-block}

With help of the EWM, we find out that the $k$-block concept with $k = 3$ helps to decrease the alphabet size for parallel addition in Eisenstein base with integer alphabets, from $\A_E = 7$ down to $\A_{E_3} = 5$. This result is obtained for the symmetric case $\A_{E_3} = \{ -2, \ldots, 2 \}$ as well as for the non-negative case $\A_{E_3} = \{ 0, \ldots, 4 \}$.

\subsubsection{Eisenstein Base with Complex Alphabet: $3$-block}

In the tested cases below, the alphabets~$\A$ are selected as subsets of $\{ 0, \pm 1, \pm \omega, \pm \omega^2 \}$.

\begin{itemize}

    \item {\bf $2$-block function on $4$-digit alphabets}: All cases pass successfully via Phase~1, but then fail in Phase~2. There are always (quite many) elements $b \in \widetilde{\A}_E + \widetilde{\A}_E$ which do no pass the $\Q[b^m]$-test.

    \item {\bf $3$-block function on $4$-digit alphabets}: All cases fail already in Phase~1; due to the fact that (quite many) representatives of the congruence classes $\mod (\widetilde{\beta} - 1)$ are missing in the alphabet $\widetilde{\A}_E$.

    \item {\bf $2$-block function on $5$-digit alphabets}: All cases pass successfully via Phase~1, but then fail in Phase~2. Already the $\Q[b^m]$-test of Phase~2 is never passed successfully, although in some cases for one element $b \in \widetilde{\A}_E + \widetilde{\A}_E$ only.

    \item {\bf $3$-block function on $5$-digit alphabets}: Here we obtain a successful result. There are 15 possible combinations of 5-digit subsets in $\A_E$ (containing $0$), and for 9 of them (depicted on Figure~\ref{Eisenstein_complex_3-block_5-digit_picture}), parallel addition can be performed via a 3-block $p$-local function. The memory parameter $r$ ranges from 2 to 4; and, with the shortest possible memory $r = 2$, we have:

        \begin{itemize}
            \item $p = r + 1 = 3$, and thus a $3$-block $3$-local addition function $\varphi : (\A_{E3} + \A_{E3})^9 \rightarrow \A_{E3}$;
            \item using one of two alphabets $\pm \A_{E3} = \pm \{ 0, 1, \omega, -\omega, -\omega^2 \}$;
            \item with the $3$-block set $\widetilde{\A_{E3}} = \{ a_2 \beta^2 + a_1 \beta + a_0 \, | \, a_j \in \A_{E3} \}$ of size $\# \widetilde{\A_{E3}} = 72$;
            \item base $\beta_{E3} = \beta_E^3 = (\omega -1)^3 = 6\omega + 3$;
            \item weight coefficients set equal to (or a subset of) $\Q_E = \A_E + \A_E$ of size $\#\Q = 19$;
            \item and the input alphabet $\A_{E3} + \A_{E3}$ of size $\# (\A_{E3} + \A_{E3}) = 280$.
        \end{itemize}

        A full description of the weight coefficient function would comprise $(\# (\A_{E3} + \A_{E3}))^r = 280^2 = 78\,400$ values. This maximum number is diminished to $161 + 33\,320 = 161 + 119*280$, since for $161$ input digits, the weight coefficient depends on just one position ($q_j = q_j (w_j)$), and only for the remaining $119$ input digits, we have to consider also their neighbour to the right ($q_j = q_j (w_j,  w_{j-1})$).

\end{itemize}

\begin{figure}[h]
    \centering
    \includegraphics[scale=0.5]{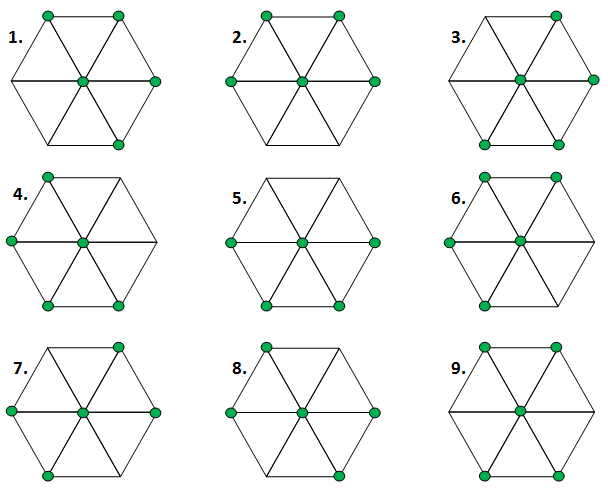}
    \caption{Here we depict the~9 subsets with~5 digits (including~$0$) in~$\A_E$, allowing parallel addition via a 3-block local function. Variants no.~1. and~4. have the minimal memory $r = 2$ of the parallel addition function.}
    \label{Eisenstein_complex_3-block_5-digit_picture}
\end{figure}

\section{Conclusions}\label{sub-sec:conclusions}

The (automated) Extending Window Method (EWM) for construction of parallel addition algorithms, as proposed and elaborated in this work, can be regarded in fact as generalization of the intuitive (manual) methods used by other authors earlier. For instance, the~EWM delivers the same parallel addition algorithms for integer bases as found previously by \cite{Avizienis}, \cite{ChowRobertson}, \cite{Parhami}, or \cite{FrPeSv2}, as illustrated in Table~\ref{tab:parallel-addition-algo_integer-bases}.

When considering certain classes of real or complex bases (quadratic, or $\ell$-th roots of integers), the EWM does provide parallel addition algorithms on (integer or non-integer) alphabets of minimal size, however, some of the underlying local functions are rather complex, even more than in some of the algorithms found manually earlier. See selected examples of real bases in Tables~\ref{tab:parallel-addition-algo_real-root-of-integer} and~\ref{tab:parallel-addition-algo_real-quadratic}, and of complex bases in Tables~\ref{tab:parallel-addition-algo_complex-quadratic} and~\ref{tab:parallel-addition-algo_complex-root-of-integer}.

In the case of complex bases and non-integer alphabets, manual search for the parallel addition algorithms is extremely laborious (if possible at all); but the~EWM could do that successfully (including block parallel addition as well), e.g. for the Penney and Eisenstein bases, samples of the Canonical Numeration Systems, as summarized in Table~\ref{tab:parallel-addition-algo_Penney-Eisenstein}.\\

Open problems remaining for future analysis regarding the Extending Window Method for construction of parallel addition algorithms are mainly the following:
\begin{itemize}
    \item Generalization of the Extending Window Method implementation using an arbitrary rewriting rule, instead of just $(X - \beta)$. This may be useful especially for numeration systems $(\beta, \A)$ with non-integer base $\beta \in \CC \setminus \ZZ$ and integer alphabet $\A \subset \ZZ$.
    \item Answering the open question of Phase 2 convergence within the Extending Window Method. This would require deeper analysis of underlying mathematical structures for the various methods considered to select the weight coefficients while iterating within Phase 2.
\end{itemize}

\section*{Acknowledgements}

The authors acknowledge financial support from the Czech Science Foundation (grant GA\v{C}R 13-03538S) and from the Czech Technical University in Prague (grant SGS 17/193/OHK4/3T/14).


\bibliographystyle{plain}

\bibliography{biblioEWM}


\end{document}